\documentclass[a4paper,10pt]{article}
\usepackage[utf8]{inputenc}
\usepackage[T1]{fontenc}
\usepackage{verbatim}
\usepackage{amssymb}
\usepackage{amsmath}
\usepackage{graphicx}
\usepackage{lmodern}
\usepackage{latexsym}
\usepackage{amsthm} 
\usepackage{here}
\usepackage{enumerate}
\usepackage{makeidx}
\usepackage{mathrsfs} %gothique
\usepackage{geometry}   % gestion facilitÃ©e des marges
\usepackage[all]{xy}  % diagrammes
\usepackage{yfonts}
\usepackage{ifsym}
\usepackage{url} 
\usepackage{cite}
\usepackage[pdftex]{hyperref}
\hypersetup{colorlinks, citecolor=black, filecolor=black,
linkcolor=black, urlcolor=black}
\usepackage{sectsty}
\allsectionsfont{\centering}
\title{Conjugacy growth series of some wreath products}
\author{Valentin Mercier}

\newcommand{\R}{\mathbb{R}}
\newcommand{\N}{\mathbb{N}}
\newcommand{\Z}{\mathbb{Z}}
\newcommand{\Co}{\mathbb{C}}
\newcommand{\Q}{\mathbb{Q}}

\newcommand{\T}{\mathscr{T}}

\newcommand{\Sph}{\textup{\textgoth{S}}}               %   Sphere in  a Cayley graph
\newcommand{\Gra}{\mathsf{Cay}}         %     Cayley graph
\newcommand{\spanned}{\textup{Span}}

\newcommand{\fix}{\textup{Fix}}
\newcommand{\supp}{\textup{supp}}

\newtheoremstyle{theorem}
{}%Space above
{}%Space below
{\itshape}%Body font
{}%Indent amount (empty = no indent, \parindent = para indent)
{\bfseries}%Thm head font
{}%Punctuation after thm head
{\newline}%Space after thm head: " " = normal interword space; \newline = linebreak
{}%Thm head spec (can be left empty, meaning `normal')

\newtheoremstyle{remark}% name
{}%Space above
{}%Space below
{\upshape}%Body font
{}%Indent amount (empty = no indent, \parindent = para indent)
{\bfseries}%Thm head font
{}%Punctuation after thm head
{\newline}%Space after thm head: " " = normal interword space; \newline = linebreak
{}%Thm head spec (can be left empty, meaning `normal')

\theoremstyle{theorem}
\newtheorem{theorem}{Theorem}[section]
\newtheorem{lemma}[theorem]{Lemma}
\newtheorem{corollary}[theorem]{Corollary}

\newtheorem{proposition}[theorem]{Proposition}

\theoremstyle{remark}
\newtheorem{example}{Example}[theorem]

\newtheorem{remark}[theorem]{Remark}

\newtheorem*{notation}{Notation}

\newtheorem*{definition}{Definition}

\newcommand\geol{\mathsf{Geo}}   % \Gamma                   geodesic language
\newcommand\geocl{\mathsf{ConjGeo}}   % \widetilde{\Gamma} geodesic conjugacy language
\newcommand\geolNorm{\mathsf{GeoNorm}}
\newcommand\conjNorm{\mathsf{ConjGeoNorm}}
\newcommand\conjrep{\mathsf{ConjRep}}

\newcommand{\sgs}{\sigma}              %                 standard growth series
\newcommand{\cgs}{\tilde{\sigma}}        %              conjugacy growth serie

\newcommand{\quotient}[2]{{\raisebox{.5em}{$#1$}\left/\raisebox{-.5em}{$#2$}\right.}}

\newcommand{\D}{\mathscr{D}}            %            Open disk in complex plane
\newcommand{\RC}{\mathcal{RC}}           %               Radius of convergence

\newcommand{\card}{\sharp}                       %  cardinal of a set

\newcommand{\Trees}{\Upsilon}     %finite trees containing the origin
\newcommand{\Treesbylefac}{\widetilde{\Upsilon}}       % set of a unique representative 
\newcommand{\Leafs}{\mathscr{L}}             %set of leafs of tree
\newcommand{\NLeafs}{\mathscr{L}^c}             % set of no leafes of a tree

\newcommand{\FuntreesG}{\Omega}

\begin{document}
\maketitle
\begin{abstract}
In this paper we consider groups of the form $G\wr L$, where the set of generators naturally extends the sets of generators of $G$ and $L$, and $L$ admits a Cayley graph that is a tree. We show how one can compute the conjugacy growth series of such groups in terms of the standard and conjugacy growth series of $G$. We then provide explicit formulas for groups of the form $G\wr \Z$ and $G\wr (C_2*C_2)$. 

We also prove that the radius of convergence of the conjugacy growth series of $G\wr L$, for any $G$ and $L$ as above, is the same as the radius of convergence of its standard growth series.

\bigskip


\bigskip


\end{abstract}
%\tableofcontents

\section{Introduction}
The conjugacy growth function of a finitely generated group $G$ records the number of conjugacy classes with a minimal length representative on the sphere of radius $n$ in the Cayley graph of $G$, for all $n\geq 0$. This was studied for free groups in \cite{R04}, \cite{C05} and \cite{R10}, for hyperbolic groups in \cite{CK02} and \cite{CK04}, for solvable groups in \cite{BdC10}, for linear groups in \cite{BCLM13}, for acylindrically hyperbolic groups in \cite{HO13} and \cite{AC17}, for certain branch groups in \cite{F14}, and several other classes of groups in \cite{GS_10}. Historically, one of the initial motivations for counting conjugacy classes of a given length came from counting closed geodesics of bounded length in compact Riemannian manifolds (see \cite{M69} for example).

In this paper we investigate conjugacy growth from both a formal and an asymptotic point of view in wreath products of the form $G\wr L$, where $L$ is a group which admits a Cayley graph that is a tree. We consider a natural generating set of $G\wr L$ built out of the standard generating sets of $G$ and $L$ (as defined in (\ref{vecY})). The formal point of view refers to computing and understanding the conjugacy growth series of these groups, that is, the complex power series with the conjugacy growth values as coefficients. In the last few years conjugacy growth series have been computed for several classes of groups based on the description of sets consisting of minimal length representatives from all conjugacy classes, as in \cite{CHHR14},\cite{CH14}, or \cite{AC17}. The computations in this paper also rely on describing the set of minimal length conjugacy representatives.

The main results of this paper are Propositions \ref{Lem_cgs_A}, \ref{Lem_cgs_Bneqe} and \ref{Lem_cgs_Be}, which together give the conjugacy growth series of a group $G\wr L$, for all groups where $L$ has a tree as its Cayley graph. These propositions lead to some explicit conjugacy growth series formulas for the groups $G\wr \Z$ and $G\wr (C_2*C_2)$ in Section \ref{Sec_examples}. We also give an asymptotic estimate of the conjugacy growth of the Lamplighter group $C_2\wr \Z$, that shows that its conjugacy growth series is transcendental over $\Q(z)$, in Proposition \ref{tran_L2}.

The asymptotic point of view refers typically to understanding the magnitude or type of the conjugacy growth function of a group, that is, whether it is exponential, intermediate or polynomial, and how it compares to the standard growth function. In \cite{GS_10} Guba and Sapir conjectured that for every amenable group of exponential growth, the conjugacy growth function is exponential, and that an amenable group with polynomial conjugacy growth function is virtually nilpotent. This conjecture can in fact be stated and is true for many classes of groups, with the notable exception of the `monster' groups of Osin \cite{O10} and Ivanov \cite[Theorem 41.2]{O91} that have a finite number of conjugacy classes but exponential standard growth, or the group in \cite{HO13} of exponential growth with two conjugacy classes. 

Another approach to measuring the asymptotic behavior of conjugacy growth is to consider the conjugacy growth rate, that is, the inverse of the radius of convergence of the conjugacy growth series. Several results show that the Guba-Sapir conjecture can be strengthened and holds for classes of groups other than the amenable ones: that is, not only is the conjugacy growth function exponential when the standard growth function is exponential, but the two functions have the same growth rate (see \cite{AC17} for the case of hyperbolic groups and \cite{CHM17} for graph products). The second main result of this paper gives an additional confirmation of this type of behavior in some finitely generated but infinitely presented groups: this is Corollary \ref{equality_rad_conv_wr} in Section \ref{Sec_cgs_A}, which shows that the radius of convergence of the conjugacy growth series of $G\wr L$ is the same as the radius of convergence of the standard growth series of $G\wr L$, for all groups where $L$ has a tree as its Cayley graph. 

%In this paper we express the conjugacy growth series of $G\wr L$ in terms of the standard and conjugacy growth series of $G$, and in Section \ref{Sec_examples} some explicit conjugacy growth series formulas for the groups $G\wr \Z$ and $G\wr (C_2*C_2)$. We prove in Section \ref{Sec_cgs_A} (Corollary \ref{equality_rad_conv_wr}) that the radius of convergence of the conjugacy growth series of $G\wr L$ is the same as the radius of convergence of the standard growth series of $G\wr L$, when $L$ has an (infinite) tree as its Cayley graph. \\

One of our main tools comes from the paper \cite{P_92}, where the author expresses the standard growth series of a group $G\wr L$, when $L$ admits a tree as its Cayley graph, in terms of the standard growth series of $G$. In order to compute the conjugacy growth in a group one needs to know when different elements are conjugate, and we use the criteria for solving the conjugacy problem developed in \cite{M_66}.

Section \ref{preliminaries} introduces the necessary prerequisites for growth series and conjugacy growth series in finitely generated groups and wreath products. In Section \ref{About_conj_lenght} we concentrate on the conjugacy growth of wreath products. We present the main differences between conjugacy classes having cursor position of infinite order, finite order, or the trivial element, where the \emph{cursor} refers to the $L$ coordinate of a group element. In Section \ref{Sec_compu_wr_tree} we describe the conjugacy growth series of the wreath products mentioned above, and in Section \ref{Sec_examples} we give explicit formulas for some examples. In the case where $L$ admits a Cayley graph which is an infinite tree and not a line, we were unable to obtain a general formula for the conjugacy growth series; the difficulty resides in counting orbits (under the left action of $L$) of finite sub-trees containing the origin in the Cayley graph of $L$, which is needed in Proposition \ref{Lem_cgs_Be}.

We point out that conjugacy growth series for some wreath products have been recently computed in \cite{HB16}, but the groups studied there are infinitely generated (unless they are finite), and different methods apply.

%%%%%%%%%%%%%%%%%%%%%%%%%%%%%%%%%%%%%%%%%%%%%%%%%%%%%%%%%%%%%%%%%%%%%%%%%%%%%%%%%%%%%%%%%%%%%%%%%%%%%%%%%%%%%%%%%%%%%%%%%%%%%%%%%%%%%%%%%%%%%%%%%%%%%%%%%%%%%%%%%%%%%%%%%%%%%%%%%%%%%%%%%%

\section{Preliminaries}\label{preliminaries}
If $X$ is a set, we write $\card X$ for its cardinal.

\subsection{Growth series and languages}\label{growt_series}
We first recall some basic facts about power series in complex analysis (see for example the reference \cite[Chapter III Section 1]{C78}).
We denote the open disc of radius $r>0$ centered at $c \in \Co$ by
$
  \D(c,r):=\{z\in\Co\,:\,|z-c|<r\},
 $
 and define a {\em complex power series} as a function $f:\D(0,r)\rightarrow \Co$ of the form
$
  f(z)=\sum_{j=0}^{\infty}a_{j}z^j,
 $
where $a_{j}\in\Co$ for all $j$. We express the fact that $a_j$ is the {\em coefficient} of $z^j$ by writing
\[
 [z^j]f(z):=a_j.
\]
The radius of convergence $\RC(f)$ of $f$ can be defined as $$
 \RC(f)=\sup\{r\in\R\,:\,f(z)\textup{ converges }\,\forall\,z\in\D(0,r)\},
$$ or equivalently as 
\[
 \RC(f)=\frac{1}{\limsup_{j\to\infty}\root j\of{|a_{j}|}}.
\]
If $\RC(f)>0$, then $f$ is defined at every point in the open disc $\D(0,\RC(f))$ and $f$ converges absolutely. In general, $f(z)$ is not defined for $z$ with $|z|= \RC(f)$.\\

\begin{proposition}\label{positive_preim_1}
 Let $f\neq 0$ be a complex power series such that $\RC(f)>0$, each coefficient $[z^j]f(z)$ is a natural number, and $[z^0]f(z)=0$. Then there exists a unique positive number $t>0$ such that $f(t)=1$ and
 \[
  t=\inf\{|z|\,:\,|f(z)|=1\}=\sup\{r>0\,:\,|f(z)|\leq 1,\,\forall z\in\D(0,r)\},
 \]
 and the infimum and supremum are attained.
\end{proposition}

\begin{proof}
Let $f=\sum_{n=1}^\infty a_n z^n$, where $a_n \in \N$ for all $n$.
On the interval $[0,\RC(f)[$ the function $f$ is strictly increasing and continuous so there exists a unique $t\in\R^{>0}$ such that $f(t)=1$. Now for any $|z|\leq t$ the following holds:
\[
 |f(z)|=|\sum_{n=1}^\infty a_n z^n|\leq \sum_{n=1}^\infty a_n |z^n|\leq \sum_{n=1}^\infty a_n t^n=f(t)=1.	
\]
\end{proof}

\begin{definition}
 Let $X$ be a set. We call the elements of $X$ {\em letters} and write $X^*$ for the set of words over $X$. A subset $L\subseteq X^*$ is called a {\em language}, and the empty word is denoted by $\epsilon$. 
\end{definition}

Let $L^n:=\underbrace{L\times\cdots\times L}_{n\textup{ times}}$ for any positive $n\in\N$, and for $l=(l_1,\ldots,l_n)\in L^n$ call the elements $l_i$, $i\in\{1,\ldots,n\}$, the {\em components} of $l$.
 Let $L$ be a language and $\simeq$ be an equivalence relation on $L$. For $l\in L$, we write $[l]_\simeq:=\{l'\in L\,:\,l'\simeq l\}$ for the {\em equivalence class} of $l$, and
 \[
  L_\simeq:=\{[l]_\simeq\,:\,l\in L\}
 \]
for the {\em quotient language of }$L$ {\em by} $\simeq$.

We now associate a complex power series to a language.
\begin{definition}
 For $w\in X^*$, the {\em length} of $w$ is the number of letters in $w$, we write it as $|w|$. If $l=(l_1,\ldots l_n)\in {X^*}^n$, the {\em length} of $l$ is defined to be $|(l_1 ,\ldots, l_n )| := \sum_{j=1}^n |l_j |$
\end{definition}

 Let $L$ be a language such that for every $m\in\N$ the set $\{l\in L\,:\,|l|=m\}$ is finite, (this for example is the case if the underlying alphabet is finite). Then we define the {\em growth series of} $L$, $F_L(z)$, to be the complex power series given by
 \[
  [z^m]F_L(z):=\card\{l\in L\,:\,|l|=m\}.
 \] 

The following is immediate.
\begin{proposition}
 Let $L$ be a language and $1\leq n\in \N$. Then
 \[
  F_{L^n}(z)=\big(F_L(z)\big)^n,\quad F_{\{u^n\,:\,u\in L\}}(z)=F_L(z^n).
 \]
\end{proposition}

\begin{notation}
 Let $L$ be a language and $\simeq$ an equivalence relation on $L$. For $[l]_\simeq\in L_{\simeq}$ we define the {\em length} of $[l]_\simeq$ to be
 \[
  |[l]_\simeq|:=\min\{|l'|\,:\,l'\in [l]_\simeq\}.
 \]
 If $L$ and $\simeq$ are such that for every $m\in\N$ the set $\{[l]_\simeq\in L\,:\,|[l]_\simeq|=m\}$ is finite, we define the {\em growth series of} $L_\simeq$, $F_{L_\simeq}(z)$, to be the complex power series given by
 \[
  [z^m]F_{L_\simeq}(z):=\card\{[l]_\simeq\in L\,:\,|[l]_\simeq|=m\}.
 \]
\end{notation}

For the practical and explicit computations of $F_{L_\simeq}(z)$, there are several approaches; one approach is to find exactly one representative $w$ in each equivalence class for the relation $\simeq$ on $L$ that satisfies $|w|=|[w]_\simeq|$, in such a way that we can easily compute the growth series of the language formed by these representative words.
Another approach is to see the equivalence classes for the relation  $\simeq$ as the orbits of a group acting on a language for which we know the growth series. In this last case we can apply Burnside's Lemma  to obtain information about the growth series $F_{L_\simeq}(z)$. %The next section gives an example of a quotient language and a computation of the associated growth series using Burnside's Lemma.
\begin{comment}
\begin{lemma}[Burnside, (see \cite{F1887} for the original proof)]
 Let $G$ be a finite group acting on a finite set $S$. The number of different orbits in $S$ under the action of $G$ is given by
 \[
  \card\,{\quotient{S}{G}}=\frac{1}{\card G}\sum_{g\in G}\card\fix(g),
 \]
where $\fix(g)$ denotes the set of elements of $S$ fixed by $g$.
\end{lemma}
\end{comment}
\subsection{Conjugacy growth series of a finitely generated group}

Let $H$ be a group with finite symmetric generating set $Z$.
For $w\in Z^*$ we write $\overline{w}$ for the corresponding element in $H$, so $\overline{\epsilon}=e$ for the trivial element of $H$. For $u,v\in Z^*$, we define $u\overline{=}v$ to hold if and only if $\overline{u}=\overline{v}$.
The {\em length} of $h\in H$ relative to $Z$ is $|h|_Z:=\min\{|w|\,:\,w\in Z^*,\,\overline{w}=h\}$, and 
for $i\in\N$, we write $$\Sph_Z(i):=\{h\in H\,:\,|h|_Z=i\}$$ for the {\em sphere of radius} $i$ in $H$.
The growth series
\[
\sgs_{(H,Z)}(z):=F_{Z^*_{\overline{=}}}(z),
\]
for which $[z^i]\sgs_{(H,Z)}(z)=\card\Sph_Z(i)$, is called the {\em standard growth series of }H (relative to $Z$).

Let us assume that no generator in $Z$ is trivial, and no two different generators represent the same element. We define the Cayley graph $\Gra(H,Z)$ of $H$, with respect to the generating set $Z$ to be the simple graph with vertex set $H$ where two vertices $h_1,h_2$ are connected by an edge if and only if there exists $q\in Z$ such that $h_1=h_2\overline{q}$.

 We write $\sim_H$ (or just $\sim$ if the context of $H$ is clear) for the equivalence relation on $H$ given by conjugation and we write $H_\sim$ for the set of conjugacy classes of $H$. For $h\in H$, we write $[h]\in H_\sim$ for the conjugacy class of $h$. Let $\overline{\sim}$ be the equivalence relation on $Z^*$ defined by $w\overline{\sim}v$ if and only if $\overline{w}\sim\overline{v}$.
  For $h\in H$ we write 
  \[
   |h|_{\sim,Z}:=\min\{|w|_Z\,:\,w\in Z^*,\,\overline{w}\in [h]\}
  \]
 for the {\em conjugacy length} of $h$ relative to $Z$, i.e. the minimum length of an element in the conjugacy class of $h$. Hence $|[h]|_{Z}:=|h|_{\sim,Z}$ is also well defined. Note that $|[h]|_{Z}=|[w]_{\overline{\sim}}|$ for the $[w]_{\overline{\sim}}\in Z^*_{\overline{\sim}}$ such that $[\overline{w}]_\sim=[h]_\sim$.  We write 
 \[
  \min([h])_Z:=\{h'\in [h]\,:\,|h'|_Z=|h|_{\sim,Z}\}
 \]
for the subset of $[h]$ consisting of elements of minimal length.

The growth series
\[
\cgs_{(H,Z)}(z):=F_{Z^*_{\overline{\sim}}}(z),
\]
for which $[z^i]\sgs_{(H,Z)}(z)=\card\{[h]_\sim\in H_\sim\,:\,|[h]|_{Z}=i\}$, is called the {\em conjugacy growth series of }H (relative to $Z$).

As mentioned in Section \ref{growt_series}, we would like to compute the growth series of the quotient languages $Z^*_{\overline{=}} $ and $Z^*_{\overline{\sim}}$ by considering the languages formed by exactly one minimal length representative word in $Z^*$ in each equivalence class. This leads us to consider the following languages. We use the notation from \cite{CHHR14}:

\begin{align*}
\geol(H,Z)&:=\{w \in Z^* \,:\, |w|=|\overline{w}|_Z\}\quad\textup{is {\em the geodesic language}},\\
\geocl(H,Z)&:=\{w \in Z^* \,:\, |w|=|\overline{w}|_{\sim,Z}\}\quad\textup{is {\em the conjugacy geodesic language}}.
\end{align*}

So $\geocl(H,Z)=\{w\in \geol(H,G)\,:\,\overline{w}\in\min([\overline{w}])_Z\}$. \\

We choose a {\em geodesic normal form}, that is, a subset $\geolNorm$ of $\geol(H,Z)$ that contains a geodesic representative word for every element of $H$. Then the growth series $\sgs_{(H,Z)}(z)$ is equal to the growth series $F_{\geolNorm}(z)$.

\begin{definition}
 Let $H$ be a group with generating set $Z$. A subset $\conjNorm\subseteq\geocl(H,Z)$ such that the map
 \[
  \xymatrix@R=.1cm{\conjNorm \ar[r]& H_\sim \\
	w\ar@{|->}[r] & [\overline{w}] }
 \]
 is a bijection
is called a set of {\em conjugacy normal forms}. The elements of the set $\{\overline{w}\,:\,w\in \conjNorm\}$ will be called the {\em conjugacy representatives}.
This means that for every conjugacy class we choose an element of smallest length in this class, and associate to it a geodesic word.
\end{definition}

Thus the growth series $\cgs_{(H,Z)}(z)$ is equal to the growth series $F_{\conjNorm}(z)$.\\

Whenever the conjugacy growth series of a group will be computed, we will explain which conjugacy normal form and conjugacy representatives we have chosen.

\begin{remark}
 Note that the notion of conjugacy growth series can be defined in any group $H$ generated by $Z$, provided that the set $\{[h]_\sim\in H_\sim\,:\,|[h]|_{\sim,Z}=i\}$ is finite, for every $i$, even if the set $Z$ is not finite. In the paper \cite{HB16} the authors considered some infinitely generated groups for which the latter sets are finite and computed the conjugacy growth series for those groups (although the standard growth series is not well-defined).
\end{remark}

\subsection{Wreath products}
We now fix a group $G$ with symmetric generating set $Y$ and neutral element $e$, and a group $L$ with symmetric generating set $X$ and neutral element $e'$.
\begin{notation}
 Let $I$ be a non-empty set. For $\eta\in \bigoplus_{i\in  I}G$ we write $\eta(i)$ for the $i^\textup{th}$ component of $\eta$, and if moreover $I$ is a group and $x\in I$, we define $\eta^x\in\bigoplus_{i\in  I}G$ by $\eta^x(i)=\eta(x^{-1}i)$, and say that $\eta^x$ is the {\em left translate} of $\eta$ by $x$.
 We write $\supp(\eta)=\{i\in I\,:\,\eta(i)\neq e\}$ for the {\em support} of $\eta$. 
\end{notation}

\begin{definition}
Let $G$ and $L$ be two groups. The {\em wreath product of }$G$ {\em by }$L$, written $G\wr L$, is defined as
\[
G\wr L:=\bigoplus_{i\in L} G \rtimes L,
\]
where for $(\eta,m),(\theta,n)\in G\wr L$, 
\[
 (\eta,m)(\theta,n)=(\eta\theta^m,mn).
\]
\end{definition}
For $(\eta,n)\in G\wr L$ we say that $n$ is the {\em cursor position}.
 For $g\in G$, let $\vec{g}\in \bigoplus_{i\in L}G$ be such that $\vec{g}(e')=g$ and $\vec{g}(i)=e$ for $i\neq e'$.
The neutral element of $G\wr L$ is then denoted by $(\vec{e},e')$, and for any $(\eta, m)\in G\wr L$, $(\eta, m)^{-1}=({(\eta^{m^{-1}})}^{-1},m^{-1})$.

The following set generates $G\wr L$:
\begin{equation}\label{vecY}
 \vec{Y}:=\{(\vec{e},x)\,:\,x\in X\}\cup\{(\vec{y},e')\,:\,y\in Y\}.
\end{equation}

\begin{remark}\label{meaning_of_tildeY}
One can interpret the generating set $\vec{Y}$ as follows. If $\Gra(L,X)$ denotes the Cayley graph of $L$ with generating set $X$, one may imagine that on each vertex $v$ of $\Gra(L,X)$ there is a copy of the group $G$.
Hence an element $(\eta,b)\in G\wr L$ can be viewed as 
\begin{enumerate}[$\bullet$]
 \item a finite subset $V$ of the set of vertices of $\Gra(L,X)$,
 \item for each $v\in V$, an element $g_v\in G\setminus\{e\}$ such that $g_v=\eta(v)$, and
 \item a vertex $v'$ of $\Gra(L,X)$ corresponding to the element $b$.
\end{enumerate}
With this interpretation, the elements of the set $\{(\vec{e},x)\,:\,x\in X\}$ induce moves of the cursor along the graph $\Gra(L,X)$, while the elements of $\{(\vec{y},0)\,:\,y\in Y\}$ can be seen as a generating set of the copy of $G$ at the vertex of $\Gra(L,X)$ corresponding to the current cursor position.
\end{remark}

\begin{proposition}[\cite{P_92}:Section 1 or \cite{M12}:Section 14.2] \label{length_in_wr_Mann} 
The length of an element $(\eta,b)$ (relative to $\vec{Y}$) is the sum of the following two values:
\begin{enumerate}
 \item The length of a minimal walk on $\Gra(L,X)$ that starts at the origin, visits every vertex $v$ for which $\eta(v)\neq e$ and ends at $b$,
 \item the sum of the lengths $|\eta(v)|_Y$ for $v\in L$.
\end{enumerate}
\end{proposition}

 Thus computing element length in a wreath product involves finding the walks of minimal length in $\Gra(L,X)$, a very difficult computational problem related to the traveling salesman problem. For this reason Parry (\cite{P_92}) only considers the case where the Cayley graph of $L$ is a tree.
\subsubsection{Description of the conjugacy classes in a wreath product}\label{conj_class_wr}
We follow the approach of \cite{M_66}, which gives a sufficient and necessary condition for two elements in $G\wr L$ to be conjugate. In \cite{M_66} the wreath product is defined via right action, so since in this paper the wreath product is defined via left action, we modify the notation and statements in \cite{M_66} accordingly. 
\begin{notation}
(1) For $(\eta,b)\in G\wr L$, let
\[
 \pi_G((\eta,b)):=\eta(b).
\]
be the entry in $\eta$ at the cursor position.

\noindent (2) Let $b\in L$ be an element and let $<b>t$ be a right coset of the subgroup $\langle b\rangle$ generated by $b$.
\begin{comment}, that is,
 \[
  L=\bigsqcup_{t\in T}\langle b\rangle t.
 \]
 \end{comment}
We define the map $\pi_{\langle b\rangle t}:\bigoplus_{l\in L} G\rightarrow G\cup G_\sim$ by
\[
 \pi_{\langle b\rangle t}(\eta):=\left\{\begin{array}{ll}
                             \left[\prod_{k=0}^{K-1}\eta(b^{-k} t)\right] & \textup{ if }b\textup{ is of finite order }K\\
                             \prod_{k=-\infty}^\infty \eta(b^{-k} t)& \textup{ else}\rlap{\,.}
                            \end{array}\right.
\]
\end{notation}
Note that this is well defined since the support of $\eta$ is finite.\\
If $\pi_{\langle b\rangle t}(\eta)=[e]\in G_\sim$, we just write it as $\pi_{\langle b\rangle t}(\eta)=e$.
\begin{comment}
\begin{remark}
 If $T'$ is another right transversal of $\langle b\rangle $ and $t\in T$ and $t'\in T'$ are such that $\langle b\rangle t=\langle b\rangle t'$, then $\pi_{\langle b\rangle t}(\eta)$ and $\pi_{\langle b\rangle t'}(\eta)$ are conjugate in the case where $b$ is of finite order and are equal otherwise.
\end{remark}
\end{comment}
Now let us follow the approach of Jane Matthews in \cite{M_66} with our notation.
\begin{comment}
Assume that $(\eta,b)$ is conjugate to $(\eta',b')$ in $G\wr L$. This means that there is $(\theta,d)\in G\wr L$ such that 
\[
 (\eta',b')=(\theta,d)^{-1}(\eta,b)(\theta,d)=({(\theta^{d^{-1}})}^{-1}\eta^{d^{-1}} \theta^{d^{-1}b},d^{-1}bd).
\]
In other words
\[
 b'=d^{-1}bd\qquad\textup{and}\qquad \eta'^{d}=\theta^{-1}\eta\theta^{b}.
\]
Fixing a right transversal $T$ of $\langle b\rangle$, this is equivalent to saying that for every $k\in\Z$ and every $t\in T$
\begin{equation}
b'=d^{-1}bd\qquad\textup{and}\qquad \eta'^{d}(b^{-k}t)=\theta^{-1}(b^{-k}t)\eta(b^{-k}t)\theta(b^{-(k+1)}t)\label{equ_conj_transversal}. 
\end{equation}
Hence taking the product over all values of $k$ in the equality (\ref{equ_conj_transversal}) gives the following.

\begin{lemma}[\cite{M_66} Lemma 3.2]\label{to_be_conj_prod}
 In $G\wr L$, the elements $(\eta,b)$ and $(\eta',b')$ are conjugate if and only if there exists $(\theta,d)\in G\wr L$ such that, $b'=d^{-1}bd$ and for any right transversal $T$ of $\langle b\rangle$, for any $ t\in T$, for all $u\in\Z$ and all $v\in \N$, we have
\begin{equation}\label{eq:Mat_66}
  \prod_{k=u}^{u+v}\eta'^{d}(b^{-k}t)=\theta^{-1}(b^{-u}t)\left( \prod_{k=u}^{u+v} \eta(b^{-k}t)\right)\theta(b^{-(u+v+1)}t).
\end{equation}
\end{lemma}
\end{comment}
\begin{lemma}[\cite{M_66} Propositions 3.5 and 3.6] \label{to_conj_anycase} %\label{to_conj_anycase} 
 Let $(\eta,b)$ and $(\eta',b')$ be in $G\wr L$. Then $(\eta,b)$ and $(\eta',b')$ are conjugate if and only if there exists $d\in L$ such that $b'=d^{-1}bd$ and for any right coset $\langle b\rangle t$
\begin{equation}\label{eq:Mat_66_infinite}
  \pi_{\langle b\rangle t}(\eta)=\pi_{\langle b\rangle t}(\eta'^d).
 \end{equation}
\end{lemma}
\begin{comment}
\begin{proof}
Suppose first that $(\eta,b)$ and $(\eta',b')$ are conjugate. Since $\langle b\rangle$ is infinite and $\supp(\theta)$, $\supp(\eta)$ and $\supp(\eta')$ are finite, there exist $u\in\Z$ and $v\in\N$ such that $\theta^{-1}(b^{-u}t)=\theta(b^{-u-v-1}t)=e$ and $\pi_{\langle b\rangle t}(\eta)=\prod_{k=u}^{u+v} \eta(b^{-k}t)$ and $\pi_{\langle b\rangle t}(\eta'^d)=\prod_{k=u}^{u+v}\eta'^{d}(b^{-k}t)$. This gives, via (\ref{eq:Mat_66}), the equality (\ref{eq:Mat_66_infinite}).

For the other direction, assume there exists $d\in L$ such that $b'=d^{-1}bd$ and for any right coset
 \[
  \pi_{\langle b\rangle t}(\eta)=\pi_{\langle b\rangle t}(\eta'^d).
 \]
 Let $T$ be a right transversal of $\langle b\rangle$. Define $\theta\in\bigoplus_{i\in L}G$ by 
 \[
  \theta(b^{-k}t):=\prod_{l=k}^{\infty}\eta(b^{-l}t)\left(\prod_{l=k}^{\infty}\eta'^d (b^{-l}t)\right)^{-1},
 \]
for every $t\in T$ and every $k\in\Z$, and verify that (\ref{equ_conj_transversal}) is satisfied.\\
Note that the hypothesis $\pi_{\langle b\rangle t}(\eta)=\pi_{\langle b\rangle t}(\eta'^d) $ assures that $\supp(\theta)$ is finite.
\end{proof}
\end{comment}

\begin{example}\label{illustration}
 Here is an illustration of Lemma \ref{to_conj_anycase} in $G\wr\Z$. Let $\Z=\langle a|\rangle$. We represent an element $\xi\in\bigoplus_{i\in\Z}G$ as a bi-infinite line on which for all $i$ such that $\xi(a^i)\neq e$ there is a symbol $\underset{i}{g}$ meaning $\xi(a^i)=g$. Let $(\eta,a^3)\in G\wr\Z$ be such that $\eta$ is represented by the following line
\[
 \xymatrix@C=1.3cm@R=.3cm{\ar@{.}[r] & \underset{-2}{g_1}\ar@{-}[r] & \underset{-1}{g_2}\ar@{-}[r] & \underset{0}{g_3} \ar@{-}[r] & \underset{1}{g_4} \ar@{-}[r] & \underset{2}{g_5} \ar@{-}[r] & \underset{3}{g_6}\ar@{-}[r] &\underset{4}{g_7} \ar@{.}[r] & \rlap{\, .}
 }
\]
Taking $T=\{e',a,a^2\}$ as right transversal of $\langle a^3\rangle$, we find 
\[
 \pi_{\langle a^3\rangle e'}(\eta)=g_6g_3,\quad \pi_{\langle a^3\rangle a}(\eta)=g_7g_4g_1,\quad \pi_{\langle a^3\rangle a^2}(\eta)=g_5g_2.
\]
Lemma \ref{to_conj_anycase} shows that $(\eta,a^3)$ is conjugate in $G\wr \Z$ to $(\zeta,a^3)$, where $\zeta$ is represented as:
\[
 \xymatrix@C=1.3cm@R=.3cm{\ar@{.}[r] & \underset{-2}{e}\ar@{-}[r] & \underset{-1}{e}\ar@{-}[r] & \underset{0}{g_6g_3} \ar@{-}[r] & \underset{1}{g_7g_4g_1} \ar@{-}[r] & \underset{2}{g_5g_2} \ar@{-}[r] & \underset{3}{e}\ar@{-}[r] &\underset{4}{e} \ar@{.}[r] & \rlap{\, .}}
\]
The conjugator is $(\theta,e')$, where $\theta$ is represented as:
\[
 \xymatrix@C=1.3cm@R=.3cm{\ar@{.}[r] & \underset{-2}{g_1}\ar@{-}[r] & \underset{-1}{g_2}\ar@{-}[r] & \underset{0}{g_6^{-1}} \ar@{-}[r] & \underset{1}{g_7^{-1}} \ar@{-}[r] & \underset{2}{e} \ar@{-}[r] & \underset{3}{e}\ar@{-}[r] &\underset{4}{e} \ar@{.}[r] & \rlap{\, .}}
\]
\end{example}

\begin{comment}
\begin{lemma}[\cite{M_66} Proposition 3.6] \label{to_conj_anycase} 
 Let $(\eta,b)$ and $(\eta',b')$ be in $G\wr L$ with $b$ of finite order $K$. Then $(\eta,b)$ and $(\eta',b')$ are conjugate if and only if there exists $d\in L$ such that $b'=d^{-1}bd$ and for every right coset $\langle b\rangle t$, $\pi_{\langle b\rangle t}(\eta)=\pi_{\langle b\rangle t}(\eta'^d)$.
\end{lemma}

\begin{proof}
Let $T$ be a right transversal of $\langle b\rangle $ in $L$.\\
Suppose $\pi_{\langle b\rangle t}(\eta)= \pi_{\langle b\rangle t}(\eta'^d)$. Since $b^{K}=e$, it suffices to take $u=0$ and $v=K-1$ in Lemma \ref{to_be_conj_prod} to show that $(\eta,b)$ and $(\eta',b')$ are conjugate.

Now assume there exists $d\in L$ such that $b'=d^{-1}bd$ and for any $t\in T$ there exists $\alpha_t\in G$ such that
  \[
   \prod_{k=0}^{K-1}\eta'^d(b^{-k} t)=\alpha_t \prod_{k=0}^{K-1}\eta(b^{-k} t)\rangle \alpha_t ^{-1}.
  \]
Define $\theta\in\bigoplus_{i\in L}G$ by 
\[
 \theta(b^{-k}t)=\left(\prod_{l=k}^{K-1} \eta(b^{-l}t)\right)\alpha_t^{-1}\left(\prod_{l=k}^{K-1} \eta'^{d}(b^{-l}t)\right)^{-1} ,
\]
for every $t\in T$ and every $k\in\{0,\ldots,K-1\}$, and verify that (\ref{equ_conj_transversal}) is satisfied.
\end{proof}
\end{comment}

\begin{corollary}\label{conj_rep_supp_singl}
 Let $(\eta,b)\in G\wr L$, let $T$ be a right transversal of $\langle b\rangle$ in $L$ and let $\{t_1,\ldots,t_k\}\subset T$ be the set of elements of $T$ such that $\pi_{\langle b\rangle t_j}(\eta)\neq e$ for $j\in\{1,\ldots,k\}$. Then $(\eta,b)$ is conjugate to an element $(\eta',b)$ for which $\supp(\eta')\cap \langle b\rangle t_j$ is a singleton $\{l_j\}$ for each $j\in\{1,\ldots,k\}$. Moreover,
 \begin{enumerate}[A)]
  \item $\eta'(l_j)=\pi_{\langle b\rangle t_j}(\eta)$ if $b$ is of infinite order, and
  \item $\eta'(l_j) \in \pi_{\langle b\rangle t_j}(\eta)$ if $b$ is of finite order.  
 \end{enumerate}
\end{corollary}

\begin{proof}
 This is an immediate consequence of Lemma \ref{to_conj_anycase}. Indeed choose $d=e'$ and for every $j\in \{1,\ldots,k\}$ choose some $l_j\in\langle b\rangle t_j$. Define $\eta'$ as $\eta'(l_j)=\pi_{\langle b\rangle t_j}(\eta)$ if $b$ is of infinite order, $\eta'(l_j) \in \pi_{\langle b\rangle t_j}(\eta)$ if $b$ is of finite order, and $\eta'(l)=e$ if $l\notin\{l_1,\ldots,l_k\}$.
 
Then $(\eta',b)$ satisfies the hypotheses of Lemma \ref{to_conj_anycase}, and $\eta'$ satisfies the corollary.
\end{proof}

%%%%%%%%%%%%%%%%%%%%%%%%%%%%%%%%%%%%%%%%%%%%%%%%%%%%%%%%%%%%%%%%%%%%%%%%%%%%%%%%%%%%%%%%%%%%%%%%%%%%%%%%%%%%%%%%%%%%%%%%%%%%%%%%%%%%%%%%%%%%%%%%%%%%%%%%%%%%%%%%%%%%%%%%%%%%%%%%%%%%%%%%%%

\section{Conjugacy class length and conjugacy growth series in $G\wr L$}\label{About_conj_lenght}
From now on we assume that the generating sets $X$ and $Y$ of $L$ and $G$, respectively, are finite. This section explains how one can shorten an element in $G\wr L$ by conjugation until it is of minimal length in its conjugacy class. We use Lemma \ref{to_conj_anycase}. Since the order of the cursor position, finite or infinite, influences the computations, we make the following distinction.

The elements $(\eta,b)$ in $G\wr L$ with $b$ of infinite order are called of {\em type} A. The words $w$ over a generating set of $G\wr L$ satisfying $\overline{w}$ is of type A are called of {\em type} A, and a conjugacy class of an element of type $A$ is also called of {\em type} A. We denote by $(G\wr L)^A$ the set of elements of $G\wr L$ of type $A$ and by $(G\wr L)_\sim ^A$ the set of conjugacy classes of type $A$. The contribution to the conjugacy growth series of $G\wr L$ of the elements of type $A$ will be denoted by $\cgs_{(G\wr L,\vec{Y})}^A$.

The elements $(\eta,b)$ in $G\wr L$ with $b$ of finite order are called of {\em type} B, and similarly use type B for all other items mentioned above.
Hence $$\cgs_{(G\wr L,\vec{Y})}=\cgs_{(G\wr L,\vec{Y})}^A+\cgs_{(G\wr L,\vec{Y})}^B.$$

\begin{definition}
Let $(\eta,b)\in G\wr L$, let $T$ be a right transversal of the subgroup $\langle b \rangle$ of $L$ and let $t_1,\ldots,t_k\in T$ be the elements of $T$ such that $\pi_{\langle b\rangle t_j}(\eta)\neq e$ for $j\in\{1,\ldots,k\}$.

\begin{enumerate}
  \item An {\em optimal coset walk for} $(\eta,b)$ is a walk of smallest possible length on $\Gra(L,X)$ starting at the origin, visiting some point $l_j \in \langle b \rangle t_j$ for all $1\leq j \leq k$ and ending at $b$.
\item An {\em optimal conjugacy walk for} $(\eta,b)$ is an optimal coset walk for $(\eta^{d^{-1}},d^{-1}bd)$ that is of minimal length among all optimal coset walks for $(\eta^{d'^{-1}},d'^{-1}bd')$ with $d'\in L$. \end{enumerate}
\end{definition}

Note that if $b=e'$ then an optimal coset walk for $(\eta,b)$ is equivalent to a minimal walk on $\Gra(L,X)$ as in Proposition \ref{length_in_wr_Mann}.

\begin{lemma}\label{min_conj_length_wr}
 Let $(\eta,b)\in G\wr L$ and let $\langle b\rangle t_1,\ldots,\langle b\rangle t_k$ be the right cosets of $\langle b\rangle$ in $L$ such that $\pi_{\langle b\rangle t_i}(\eta)\neq e$ for $i\in\{1,\ldots,k\}$. Then there is a conjugacy representative $(\zeta,d^{-1}bd)$ of $[(\eta,b)]$
 with the following properties. 
 \begin{enumerate}
  \item An optimal coset walk for $(\zeta,d^{-1}bd)$ is an optimal conjugacy walk for $(\eta,b)$.
  \item \label{singleton_for_coset}If such a walk intersects $d^{-1}\langle b\rangle t_i$ at $d^{-1}b^{a_i} t_i$, where $a_i\in\Z$, $i\in\{1,\ldots,k\}$, then $\supp(\zeta)=\{d^{-1}b^{a_1} t_1,\ldots,d^{-1}b^{a_k} t_k\}$ and 
  \begin{enumerate}[A)]
   \item $\zeta(d^{-1}b^{a_i} t_i)=\pi_{\langle b\rangle t_i}(\eta)$ if $b$ is of infinite order,
   \item $\zeta(d^{-1}b^{a_i} t_i)\in\min(\pi_{\langle b\rangle t_i}(\eta))_Y$ if $b$ is of finite order.
  \end{enumerate}
 \end{enumerate}
\end{lemma}

\begin{proof}
For all $i\in\{1,\ldots,k\}$ let $l_i\in \langle b\rangle t_i$ be such that an optimal coset walk for $(\eta,b)$ passes through $l_i$. Corollary \ref{conj_rep_supp_singl} implies that $(\eta,b)$ is conjugate to an element $(\xi,b)$ of smaller or equal length that satisfies $\supp(\xi)=\{l_1,\ldots,l_k\}$, and
 for all $i\in\{1,\ldots,k\}$, if (i) $b$ is of infinite order, then $\xi(l_i)=\pi_{\langle b \rangle t_i}(\eta)$, and if (ii) $b$ is of finite order, then $\xi(l_i)\in\min(\pi_{\langle b \rangle t_i}(\eta))_Y$.

In this case the length of $(\xi,b)$ is equal to the sum of
\begin{enumerate}
 \item the length of an optimal coset walk of $(\eta,b)$ on $\Gra(L,X)$, and
 \item $\sum_{i=1}^k|\pi_{\langle b \rangle t_i}(\eta)|_Y$.
 \begin{comment}
 \item 	\begin{enumerate}[A)]
         \item $\sum_{i=1}^k|\pi_{\langle b \rangle t_i}(\eta)|_Y$ if $b$ is of infinite order, 
         \item $\sum_{i=1}^k|[\pi_{\langle b \rangle t_i}(\eta)]|_Y$ if $b$ is of finite order.
        \end{enumerate}
        \end{comment}
\end{enumerate}
A consequence of Lemma \ref{to_conj_anycase} is that inside the conjugacy class of $(\eta,b)$ in $G\wr L$ no element $(\zeta, c)$ has the value $\sum_{l\in\supp(\zeta)}|\zeta(l)|_Y$ strictly smaller then $\sum_{l\in\supp(\xi)}|\xi(l)|_Y$.\\

Now we can also translate the support of $\xi$ along $\Gra(L,X)$ and modify the cursor position in order to minimize the walk on $\Gra(L,X)$. More precisely, taking a suitable $d\in L$ and conjugating $(\xi,b)$ by $(\vec{e},d)$ we find the element $(\zeta,d^{-1}bd):=(\vec{e},d)^{-1}(\xi,b)(\vec{e},d)=(\xi^{d^{-1}},d^{-1}bd)$. Hence this $\zeta$ satisfies $\zeta^d=\xi$ and $\supp(\zeta)=\{d^{-1}l_1,\ldots,d^{-1}l_k\}$.\\
\end{proof}

\begin{corollary}
The conjugacy length of $(\eta,b)$ is equal to the sum of
\begin{enumerate}
 \item the length of an optimal conjugacy walk for $(\eta,b)$, and
 \item $\sum_{i=1}^k|\pi_{\langle b \rangle t_i}(\eta)|_Y$.
 \begin{comment}
 \item \begin{enumerate}[A)]
         \item $\sum_{i=1}^k|\pi_{\langle b \rangle t_i}(\eta)|_Y$ if $b$ is of infinite order, 
         \item $\sum_{i=1}^k|[\pi_{\langle b \rangle t_i}(\eta)]|_Y$ if $b$ is of finite order.
        \end{enumerate}
        \end{comment}
\end{enumerate}
\end{corollary}

\begin{example}
Let $(\eta,a^3)$ be as in Example \ref{illustration}. We choose $T=\{e',a,a^2\}$ as the right transversal of $\langle a^3\rangle$ in $\Z$, and let $t_1:=e'$, $t_2=a$ and $t_3:=a^2$ be the elements of $T$ such that $\pi_{\langle a^3\rangle t_i}(\eta)\neq e$. Since $\Z$ is abelian and the direct walk on $\Gra(\Z,\{a,a^{-1}\})$ from $e'$ to $a^3$ visits every coset $\langle a^3\rangle t_j$, this is the only optimal coset walk for $(\eta,a^3)$.

The optimal conjugacy walk for $(\eta,a^3)$ starts at $e'$ and creates $g_6g_3$, goes to $a$ and creates $g_7g_4g_1$, then goes to $a^2$ and creates $g_5g_2$, and ends up at $a^3$. This is an optimal walk for the element $(\zeta,a^3)$ for which $\supp(\zeta)\subset\{e',a,a^2\}$ as in Example \ref{illustration}. In this case the $d$ and the $a_i$'s in Lemma \ref{min_conj_length_wr} are all trivial and $l_1=e$, $l_2=a$ and $l_3=a^2$. The length of the walk on $\Gra(\Z,\{a,a^{-1}\})$ is 3 and the sum the length of the components is $|g_6g_3|_Y+|g_7g_4g_1|_Y+|g_5g_2|_Y$, hence  $|[(\eta,a^3)]|_{\vec{Y}}=3+|g_6g_3|_Y+|g_7g_4g_1|_Y+|g_5g_2|_Y$. Note that $(\eta,b^3)$ is also conjugate to the element $(\zeta',a^3)\in\min([(\eta,a^3)])$ where $\zeta'$ is represented as
\[
 \xymatrix@C=1.3cm@R=.3cm{\ar@{.}[r] & \underset{-2}{e}\ar@{-}[r] & \underset{-1}{e}\ar@{-}[r] & \underset{0}{e} \ar@{-}[r] & \underset{1}{g_7g_4g_1} \ar@{-}[r] & \underset{2}{g_5g_2} \ar@{-}[r] & \underset{3}{g_6g_3}\ar@{-}[r] &\underset{4}{e} \ar@{.}[r] & \rlap{\, .}}
\]
In this case the conjugator is $(\theta',e')$, where $\theta'$ is represented as:
\[
 \xymatrix@C=1.3cm@R=.3cm{\ar@{.}[r] & \underset{-2}{g_1}\ar@{-}[r] & \underset{-1}{g_2}\ar@{-}[r] & \underset{0}{g_3} \ar@{-}[r] & \underset{1}{g_7^{-1}} \ar@{-}[r] & \underset{2}{e} \ar@{-}[r] & \underset{3}{e}\ar@{-}[r] &\underset{4}{e} \ar@{.}[r] & \rlap{\, .}}
\]

\end{example}

Thus finding the conjugacy length of an element involves finding a conjugate of the cursor position which produces an optimal walk on $\Gra(L,X)$.
We will see in the next section that when $\Gra(L,X)$ is a tree, a $(\zeta,d^{-1}bd)$ as in Lemma \ref{min_conj_length_wr} can be chosen such that $d^{-1}bd$ is minimal in its conjugacy class (in $L$), hence cyclically reduced if $b$ is of infinite order.

\subsection{The case when the Cayley graph of $L$ is a tree}
From now on, we assume that $\Gra(L,X)$ is a tree. This is precisely the case when $L$ admits the presentation 
\[
 L=\langle a_1,\ldots, a_M,b_1,\ldots,b_N\,|\,b_1^2,\ldots,b_N ^2\rangle.
\]
Hence we assume that $X=\{a_1,a_1 ^{-1},\ldots,a_M,a_M^{-1},b_1,\ldots,b_N\}$ and each generator $b_j$ is equal to its inverse. This implies that $\Gra(L,X)$ is the $(2M+N)$-regular tree.

In this case $\geolNorm(L,X)=\geol(L,X)$ is the set of freely reduced words on $X$. For simplicity we will make no distinction between elements of $L$ and freely reduced words, and by a word on $X$, we will mean freely reduced word.

An element of $L$ is minimal in its conjugacy class in $L$ if and only if it is cyclically reduced or equal to one of the torsion generators $b_j$ for some $j\in\{1,\ldots,N\}$ or is trivial, so the set of conjugacy geodesics $\geocl(L,X)$ is given by
\[
 \geocl(L,X)=\geocl(L,X)^A\cup\{b_1,\ldots,b_N\}\cup\{\epsilon\},
\]
where $\geocl(L,X)^A$ denotes the set of cyclically reduced words, whose growth series computation is similar to that in \cite[Theorem 1.1]{R10}.

\begin{lemma}[see Theorem 1.1 \cite{R10}]\label{cyl_red_A} 
  In the group $L$ as above the number of cyclically reduced words of length $k>0$ is 
  \[
    (2M+N-1)^k+(-1)^k(M+N-1)+M,
  \]
and the associated growth series written $F_{\geocl(L,X)^A}(z) $ is given by
\[
 F_{\geocl(L,X)^A}(z)=\frac{1}{1-(2M+N-1)z}+\frac{(2M+N)z^2-(N-1)z-1}{1-z^2}.
\]
\end{lemma}
The proof is similar to the proof of Theorem 1.1 \cite{R10}. In \cite{R10} Rivin counts the number of cyclically reduced words of length $k$ in a free group by counting the number of closed paths of length $k$ in a graph having the generators (and their inverses) as vertices, where each vertex is connected to all other vertices except for its inverse. He then computes the trace of the $k^\textup{th}$ power of the adjacency matrix. We generalize his result by also adding a vertex for each generator of order $2$ that is connected to all the vertices excepted itself. We compute the trace of the $k^\textup{th}$ power of the adjacency matrix in the same way as Rivin.
\begin{remark}\label{RC_cycl_redu}
If $2M+N\geq 2$, then 
 \[
  \RC(F_{\geocl(L,X)^A}(z))=\frac{1}{2M+N-1}.
 \]

\end{remark}

Now let us adapt Lemma \ref{min_conj_length_wr} to this kind of group.
\begin{lemma}\label{min_conj_length_wr_tree}
 Let $(\eta,b)\in G\wr L$ be such that $b\neq e'$. Then there exists $(\zeta,c)\in\min([(\eta,b)])_{\vec{Y}}$ such that either $c$ is of infinite order and cyclically reduced, or $c$ is one of the torsion generators $b_j$ of $L$ for some $j\in\{1,\ldots, N\}$. %Hence b is minimal in its conjugacy class in $L$.
  \end{lemma}
\begin{proof}
 Assume that $b$ is not cyclically reduced and $b\neq b_j$, for any $j\in\{1,\ldots,N\}$. We consider the two cases, when $b$ is of infinite order, and when $b$ is of finite order, separately. In both cases Lemma \ref{min_conj_length_wr} shows that $(\eta,b)$ is conjugate to an element $(\eta',b)$ of the same length for which $\supp(\eta')$ intersects each right coset of $\langle b\rangle$ at most once and an optimal walk for this element is an optimal coset walk. Hence up to conjugating $(\eta,b)$, one can assume that $(\eta,b)$ is of this form. We consider a particular suitable coset walk for $(\eta,b)$ and then after conjugation we show that $(\eta,b)$ is conjugate to an element $(\zeta,c)\in\min([(\eta,b)])_{\vec{Y}}$ with $c$ as claimed.
 \begin{enumerate}[A)]
  \item 
  
  If $b$ is of infinite order it can be written as $\overline{w}^{-1}\overline{c}\overline{w}$, where $w=w_1\ldots w_n$ and $c=c_1 \ldots c_k$ are in $X^*$ and freely reduced, and $w_1\neq c_1\neq {c_k}^{-1}\neq {w_1}$. For simplicity of notation, we write $w$ for $\overline{w}$ and $c$ for $\overline{c}$.
  
   Let us consider an optimal walk on $\Gra(L,X)$ starting at the origin, visiting each right coset $\langle b\rangle t$ of $\langle b\rangle$ such that $\pi_{\langle b\rangle t}(\eta)\neq e$, and ending at $b$. 
 For $r=\{0,\ldots,k-1\}$ let $T_r$ be the set of (freely reduced) words in $X$ that do not begin with $c_r^{-1}$ or $c_{r+1}$, where $c_0:=c_k$. Then we can assume that the optimal walk on $\Gra(L,X)$ starts at the origin, goes inside $w^{-1}T_0$ and comes back to $w^{-1}$, then goes to $w^{-1}c_1$ and inside $w^{-1}c_1T_1$, comes back to $w^{-1}c_1$, until it reaches $w^{-1}c_1\cdots c_{k-1}$, goes inside $w^{-1}c_1\cdots c_{k-1}T_{k-1}$, comes back to $w^{-1}c_1\cdots c_{k-1}$ and finally stops at $w^{-1}cw$. In other words, we assume
 \[
  \supp(\eta)\subset w^{-1}T_0\cup w^{-1}c_1 T_1 \cup\ldots \cup w^{-1}c_1\cdots c_{k-1} T_{k-1},
 \]
and that the walk on $\Gra(L,X)$ is as Figure \ref{supportNU}:

\begin{figure}[H]
 \[
  \xymatrix@R=.7cm@C=.2cm{ &&&&&  & & & & & & & & & & & &  &&&&&&\\
 &&&&& w^{-1}T_0 &  & & &w^{-1}c_1 T_1 & & & & & w^{-1}c_1\cdots c_{k-1} T_{k-1} & & & &&&&&&\\
 \underset{e'}{\bullet} \ar@{-}[rr]&&\ar@{.}[rr]&&\ar@{-}[rr] &  & \underset{w^{-1}}{\bullet} \ar@{.}[uu]\ar@{-}[ruu]\ar@{-}[rrr] &&& \underset{w^{-1}c_1}{\bullet}\ar@{-}[luu]\ar@{.}[uu]\ar@{-}[ruu]\ar@{.}[rrrrr] &&&&& \underset{w^{-1}c_1\cdots c_{k-1}}{\bullet}  \ar@{-}[luu] \ar@{.}[uu] \ar@{-}[ruu]\ar@{-}[rrr] &&& \underset{w^{-1}c}{\bullet}\ar@{-}[rr] &&\ar@{.}[rr]&&\ar@{-}[rr]&&\underset{w^{-1}c w}{\bullet}\rlap{\,.}
  }
 \]
 \caption{Illustration of $(\eta,w^{-1}cw)$}\label{supportNU}
 \end{figure}
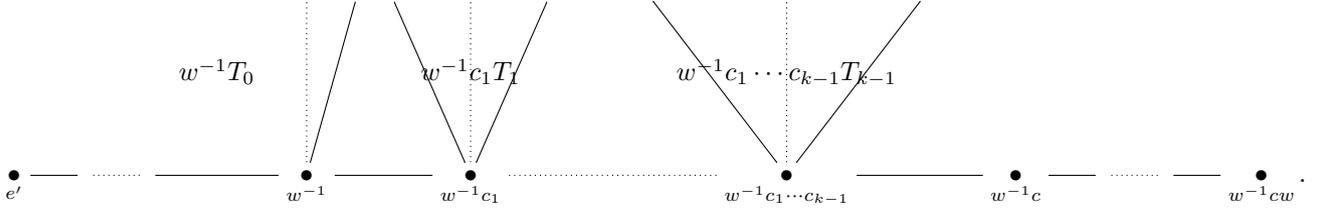
For $r\in\{0,\dots,k-1\}$, let $l_r$ be the length of the walk in $w^{-1}c_1\cdots c_r T_r$. This is 2 times the number of edges in the tree spanned by $\supp(\eta)\cap w^{-1}c_1\cdots c_r T_r$ for $r\in\{1,\ldots,k-1\}$ and at most 2 times the number of edges in the tree spanned by $\supp(\eta)\cap w^{-1}T_0+2n$ for $r=0$. Note that $l_r$ can be equal to 0. Then the length of the walk on $\Gra(L,X)$ is $\sum_{r=0}^{k-1}l_r + 2n+k$.

Now let us conjugate $(\eta,b)$ by $(\vec{e},w)$ to obtain $(\zeta,c)$. In this case $\zeta^{w^{-1}}=\eta$ and hence $\supp(\zeta)=w\supp(\eta)$. Then an optimal walk on $\Gra(L,X)$ starting at the origin, visiting every element of $\supp(\zeta)$ and ending at $c$ has the following trajectory: it first goes inside $T_0$, then comes back to $e'$, then goes to $c_1$, inside $c_1 T_1$, comes back to $c_1$, and so on until it reaches $c_1\cdots c_{k-1}$, walks inside $c_1\cdots c_{k-1}T_{k-1}$, comes back to $c_1\cdots c_{k-1}$ and finally stops at $c$, as shown in Figure \ref{Better}:
\begin{figure}[H]
 \[
  \xymatrix@R=.7cm@C=.4cm{   & & & & & & & & & & & & \\
  & T_0 & & & c_1 T_1 & & & & & c_1\cdots c_{k-1}T_{k-1} & & & \\
   & \underset{e'}{\bullet} \ar@{-}[luu]\ar@{.}[uu]\ar@{-}[ruu]\ar@{-}[rrr] &&& \underset{c_1}{\bullet}\ar@{-}[luu]\ar@{.}[uu]\ar@{-}[ruu]\ar@{.}[rrrrr] &&&&& \underset{c_1\cdots c_{k-1}}{\bullet}  \ar@{-}[luu] \ar@{.}[uu] \ar@{-}[ruu]\ar@{-}[rrr] &&& \underset{c}{\bullet}\rlap{\,.}
  }
 \] 
 \caption{Illustration of $(\zeta,c)$}\label{Better}
\end{figure}
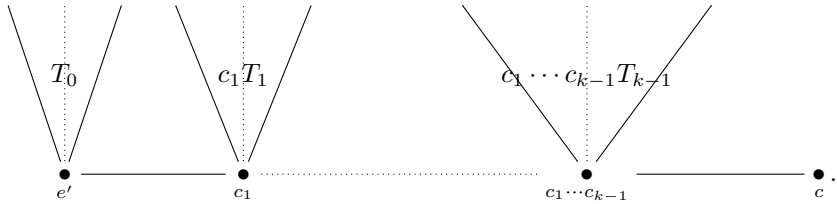

 This is an optimal coset walk for $(\zeta,c)$. Therefore the length of the walk on $\Gra(L,X)$ is less or equal to $\sum_{r=0}^{k-1}l_r +k+2n$. On the other hand, the value $\sum_{v\in\supp(\zeta)}|\zeta(v)|_Y$ is the same as the value $\sum_{v\in\supp(\eta)}|\eta(v)|_Y$. Therefore the length of $(\zeta,c)$ is less or equal to the length of $(\eta,b)$ as claimed.
 \item
  If $b$ is of finite order it can be written as $\overline{w}^{-1}\overline{b_j}\overline{w}$ for some $j\in\{1,\ldots,N\}$, where $w=w_1\cdots w_n$ is in $X^*$ and freely reduced and with $w_1\neq b_j$.\\
 
 Let $T$ be the set of (freely reduced) words in $X$ that do not begin with $b_j$. Then as before we can assume that the optimal walk on $\Gra(L,X)$ starts at the origin, goes inside $w^{-1}T$ and comes back to $w^{-1}$ then goes to $w^{-1}b_j$ and goes inside $w^{-1}b_jT$, comes back to $w^{-1}b_j$ and finally stops at $w^{-1}b_jw$. In other words we assume
 \[
  \supp(\eta)\subset w^{-1}T,
 \]
and that the walk on $\Gra(L,X)$ is as represented in Figure \ref{nonNiceB}:
\begin{figure}[H]
 \[
  \xymatrix@R=.7cm@C=.3cm{ &&&&&  & & &   &&&&&&\\
 &&&&&w^{-1}T &  & & &   & & &&&&\\
 \underset{e'}{\bullet} \ar@{-}[rr]&&\ar@{.}[rr]&&\ar@{-}[rr] &  & \underset{w^{-1}}{\bullet} \ar@{.}[uu]\ar@{-}[ruu]\ar@{-}[rrr] &&&  \underset{w^{-1}b_j}{\bullet}\ar@{-}[rr] &&\ar@{.}[rr]&&\ar@{-}[rr]&&\underset{w^{-1}b_j w}{\bullet}\rlap{\,.}
  }
 \]
 \caption{Illustration of $(\eta,w^{-1}b_jw)$}\label{nonNiceB}
 \end{figure}
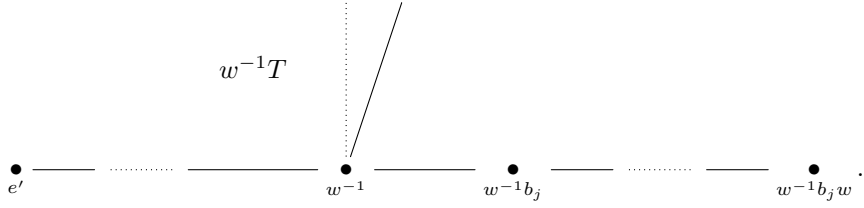
 
 In the same way as before $(\eta,b)$ is conjugate to a $(\zeta,b_j)$ with $\supp(\zeta)=w\supp(\eta)$, as represented in Figure \ref{NiceB};
 \begin{figure}[H]
 \[
  \xymatrix@R=.7cm@C=.4cm{  & & &   \\
  & T & & &   \\
   & \underset{e'}{\bullet} \ar@{-}[luu]\ar@{.}[uu]\ar@{-}[ruu]\ar@{-}[rrr] &&&  \underset{b_j}{\bullet},
  }
 \]
  \caption{Illustration of $(\zeta,b_j)$}\label{NiceB}
 \end{figure}
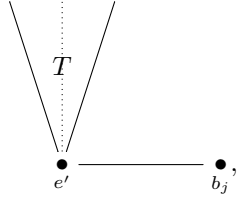
and the length of $(\zeta,b_j)$ is smaller or equal to the length of $(\eta,b)$, since $\sum_{v\in\supp(\zeta)}|[\zeta(v)]|_Y$ is the same as the value $\sum_{v\in\supp(\eta)}|[\eta(v)]|_Y$. Therefore we also find $(\zeta,b_j)$ as claimed.
\end{enumerate}
This proves the lemma. 
\end{proof}

The previous proof also shows the following.
\begin{corollary}\label{equiv_conj_non_triv_curs}
Let $(\eta,c)\in G\wr L$ be an element with $c\neq e'$ and such that $(\eta,c)\in\min([(\eta,c)])_{\vec{Y}}$. Using the same notation as above the following holds.
\begin{enumerate}[A)]
 \item If $c=c_1\cdots c_k$ is cyclically reduced and we choose a conjugacy representative $(\eta,c)$ that satisfies
\[
  \supp(\eta)\subset T_0\cup c_1 T_1 \cup\ldots \cup c_1\cdots c_{k-1} T_{k-1},
 \]
then such an element is uniquely determined up to cyclic permutation of $c$ and the components of $\eta$ in the trees $T_r$. 
\item If $c=b_j$ for some $j\in\{1,\ldots,N\}$ and for every conjugacy class in $G$ we choose a unique conjugacy representative, then there is a unique conjugacy representative $(\eta,b_j)$ that satisfies
\[
  \supp(\eta)\subset T.
 \]
\end{enumerate}
\end{corollary}

Now one needs an analogue to Lemma  \ref{min_conj_length_wr_tree} for elements having trivial cursor position.
\begin{lemma}\label{min_conj_length_wr_triv_curs}%\label{equiv_conj_triv_curs}
 Let $(\eta,e')\in G\wr L$ be non-trivial and such that $(\eta,e')\in\min([(\eta,e')])_{\vec{Y}}$. Then the finite tree in $\Gra(L,X)$ spanned by $\supp(\eta)$ contains the vertex $e'$.
 
 Moreover, two elements $(\eta,e')\in\min([(\eta,e')])_{\vec{Y}}$ and $(\eta',e')\in\min([(\eta',e')])_{\vec{Y}}$ are conjugate if and only if there exists $d\in L$ such that for every $l\in L$, $\eta'(l)$ and $\eta^d(l)$ are conjugate in $G$. 
\end{lemma}
\begin{proof}
 This is an immediate consequence of Lemma \ref{min_conj_length_wr}.
\end{proof}

\section{The conjugacy growth series when the Cayley graph of $L$ is a tree}\label{Sec_compu_wr_tree}
In this section, we compute the contribution to $\cgs_{(G\wr L,\vec{Y})}^A(z)$ of the elements $(\eta,c)$, with $|c|_X=k$, and the contribution to $\cgs_{(G\wr L,\vec{Y})}^B(z)$ of the elements $(\eta,c)$, with $|c|_X=1$ or $0$.

For the conjugacy classes of elements $(\eta,c)$, with $|c|_X=k>0$, we consider a cyclically reduced element $c=c_1\cdots c_k\in L$, and (with the notation in the proof of Lemma \ref{min_conj_length_wr_tree}) we count the contribution of all the possible $\eta$ satisfying the condition
\[
  \supp(\eta)\subset T_0\cup c_1 T_1 \cup\ldots \cup c_1\cdots c_{k-1} T_{k-1}.
 \]
 To do this we must consider the contribution of the possible walks along the subtrees $T_0,\, c_1 T_1, \,\ldots ,$ $\,c_1\cdots c_{k-1} T_{k-1}$. We will follow the approach of Parry from \cite{P_92}, where he computed the standard growth series of $G\wr L$, and use his notation. 
 \begin{notation}
  The intersection of the trajectory of the walk with each of the trees $T_0,\, c_1 T_1, \,\ldots ,\,c_1\cdots c_{k-1} T_{k-1}$ is a finite tree $T$. For such a tree $T$ define a vertex to be a leaf if it has valence 1 or 0 (when the tree contains just one vertex).
  Choose $v\in X\setminus\{c_1,c_k^{-1}\}$. Let $\T$ be the set of all finite subtrees containing $e'$ but no other vertex adjacent to $e'$ than $v$, and define
  \[
F_\T(x,y)=\sum_{i,j\geq 0}a_{ij}x^{i}y^j   
  \]
to be the formal complex power series on $2$ variables with
\[
 a_{ij}:=\card\{T\in\T\,:\; T\textup{ contains }i\textup{ non-leaves and }j\textup{ leaves other than } e'\}.
\]
 \end{notation}
Parry observed that $F_\T(x,y)$ does not depend on the edge $\{e',v\}$, but only on the degree of the regular tree (in this case $2M+N$). 
\begin{lemma}[\cite{P_92}:Lemma 3.1]
 If the degree of the tree is $D$, then $F_\T(x,y)$ satisfies
 \[
  F_\T(x,y)=1+y-x+xF_\T(x,y)^{D-1}.
 \]
\end{lemma}
In particular, if $D=1$ then $F_\T(x,y)=1+y$, and if $D=2$ then $F_\T(x,y)=1+y\frac{1}{1-x}$.
He first computes the contribution to the growth series of a walk starting at the origin with support in $T$, for a $T\in\T$, creating the non-trivial elements of $G$ at the leaves of $T$ and the other elements of $G$ in the non-leaves of $T$ and then coming back to the origin.

Let the group $L$ have presentation $\langle a_1,\ldots, a_M,b_1,\ldots,b_N\,|\,b_1^2,\ldots,b_N ^2\rangle$ and symmetric generating set $X=\{a_1,a_1 ^{-1},\ldots,a_M,a_M^{-1},b_1,\ldots,b_N\}$ and
assume the valence of the tree $\Gra(L,X)$ is at least 2, in other words $2M+N\geq 2$.
We keep the notation in the proof of Lemma \ref{min_conj_length_wr_tree} and Section \ref{About_conj_lenght}, and let $F_\T(x,y)$ refer to the $2$-variable complex power series with natural number coefficients that satisfies 
\[
  F_\T(x,y)=1+y-x+xF_\T(x,y)^{2M+N-1}.
 \]

\subsection{Contribution of conjugacy classes with cursor of infinite order}\label{Sec_cgs_A} 
 The goal of this section is to prove Proposition \ref{Lem_cgs_A}. Recall that $\cgs^A _{(G\wr L,\vec{Y})}(z)$ denotes the contribution of the conjugacy classes of type A (i.e. with cursors of infinite order) to the conjugacy growth series of $G\wr L$.
\begin{proposition}\label{Lem_cgs_A}
The formula for $\cgs^A _{(G\wr L,\vec{Y})}(z)$ is
\begin{equation}
 \cgs^A _{(G\wr L,\vec{Y})}(z)=\sum_{r\geq 1}\frac{\phi(r)}{r}\sum_{s\geq 1}\frac{(2M+N-1)^s+(-1)^s(M+N-1)+M}{s}{(F_E(z^r))}^s, \label{cgs_A}
\end{equation}
where $\phi(r)$ denotes Euler's totient function and where 
\[
 F_E(z):=z\sgs_{(G,Y)}(z)\left(F_\T(z^2\sgs_{(G,Y)}(z),z^2(\sgs_{(G,Y)}(z)-1))\right)^{2M+N-2}.
\]
\end{proposition} 
We construct a conjugacy geodesic normal form of type A as follows. For every cyclically reduced word $c$ of length $k$ we associate a geodesic normal form representing all the elements $(\eta,c)$ where the support of $\eta$ is represented in Figure \ref{NiceA2}
(see Corollary \ref{equiv_conj_non_triv_curs} A):
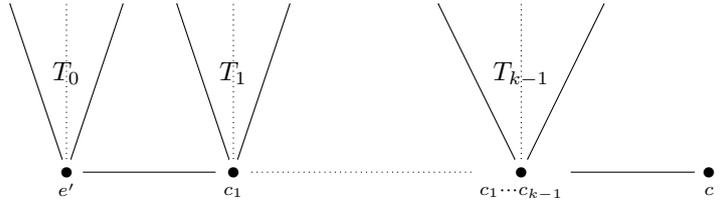
\begin{figure}[H]
\[
  \xymatrix@R=.7cm@C=.4cm{   & & & & & & & & & & & & \\
  & T_0 & & & T_1 & & & & & T_{k-1} & & & \\
   & \underset{e'}{\bullet} \ar@{-}[luu]\ar@{.}[uu]\ar@{-}[ruu]\ar@{-}[rrr] &&& \underset{c_1}{\bullet}\ar@{-}[luu]\ar@{.}[uu]\ar@{-}[ruu]\ar@{.}[rrrrr] &&&&& \underset{c_1\cdots c_{k-1}}{\bullet}  \ar@{-}[luu] \ar@{.}[uu] \ar@{-}[ruu]\ar@{-}[rrr] &&& \underset{c}{\bullet}\rlap{\,.}
  }
 \]
 \caption{Illustration of $(\eta,c)$}\label{NiceA2}
 \end{figure}
For a given $c$ of length $k$, each such geodesic normal form can be seen as the concatenation of {\em elementary blocks} $(\eta_r,c_{r+1})$, where $\supp(\eta_r)\subset T_r$, for $r\in\{0,\ldots,k-1\}$, as represented in Figure \ref{elblockA}. For such a concatenation, each vertex of each $T_r$ will be in a different coset of $\langle c\rangle$. This is why any value of $G$ can be associated to this vertex.
\begin{figure}[H]
 \[
 \xymatrix@R=.7cm@C=.4cm{ & & & & \\
  & T_r & & &\\
  &\underset{e'}{\bullet} \ar@{-}[luu]\ar@{.}[uu]\ar@{-}[ruu]\ar@{-}[rrr] & & &\underset{c_{r+1}}{\bullet}\rlap{.}}
\]
\caption{Illustration of $(\eta_r,c_{r+1})$}\label{elblockA}
\end{figure}
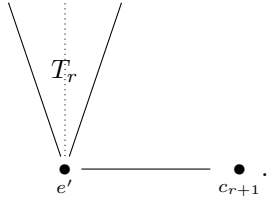

For a given $c_{r+1}$, the growth series of the geodesic normal forms representing the set of elementary blocks $\supp(\eta_r)\subset T_r$ does not depend on $r$ but only on the degree of the tree and on the growth series of $G$. This leads us to introduce the following.

\begin{notation}
For $x_1,x_2\in X$ such that $x_1\neq x_2$, we define
\begin{equation}
 E_{x_1,x_2}=\left\{(\eta,x_2)\in G\wr L\,:\,\supp(\eta)\subset\{ \textup{ elements in } L \textup{ that do not begin with } x_1 \textup{ or } x_2\}\right\}.\label{El_bloc}
\end{equation} 
\end{notation}
\begin{lemma}
 The contribution to the standard growth series of the elements of $E_{x_1,x_2}$ is given by
 \begin{equation}
 F_E(z):=z\sgs_{(G,Y)}(z)\left(F_\T(z^2\sgs_{(G,Y)}(z),z^2(\sgs_{(G,Y)}(z)-1))\right)^{2M+N-2}.\label{growth_el_bloc}
\end{equation}
\end{lemma}
\begin{proof}
 Let us compute the contribution of a walk that starts at $e'$, goes along all directions except $x_1$ and $x_2$, comes back to $e'$ and ends at $x_1$. Consider one direction $x$ among the $2M+N-2$. If the trajectory of this walk is a tree with $i$ non-leaves and $j$ leaves (other than $e'$), then the length of this walk on $\Gra(L,X)$ is $2(i+j)$, hence $z^{2(i+j)}$ is a factor of the series. On the other hand, any element in $G$ can be associated to a non-leaf, while any non-trivial element in $G$ can be associated to a leaf (different from $e'$). Therefore $\sgs_{(G,Y)}(z)^i$ and ${(\sgs_{(G,Y)}(z)-1)}^j$ are also factors of the series. Considering all the possible walks in direction $x$ gives $F_\T(z^2\sgs_{(G,Y)}(z),z^2(\sgs_{(G,Y)}(z)-1))$. Now considering all the $2M+N-2$ possible choices for $x$, we find $\left(F_\T(z^2\sgs_{(G,Y)}(z),z^2(\sgs_{(G,Y)}(z)-1))\right)^{2M+N-2}$. Finally, the value of $G$ at $e'$ is arbitrary, so $\sgs_{(G,Y)}(z)$ appears again, and the walk from $e'$ to $x_1$ produces the factor $z$. 
\end{proof}

Now, considering the language of all such geodesic normal forms for the cyclically reduced words of length $k$ will give rise to a new language; within this language we define two words to be equivalent if and only if they differ by a cyclic permutation of the $k$ elementary blocks. Then the conjugacy geodesic normal forms of type A will be the union, over all $k\in \N$, of the languages above, up to the cyclic permutation equivalence.
So let us define
\begin{equation}
 \conjrep^A(G\wr L):=\bigcup_{\stackrel{c_1\cdots c_k\in L,\,k\geq 1}{\textup{cyclically reduced}}} E_{c_k^{-1},c_1} E_{c_1^{-1},c_2}\cdots E_{c_{k-1}^{-1},c_k},\label{Conc_el_bloc}
\end{equation}
and
\begin{equation}
 {\conjrep_k} ^A(G\wr L):=\bigcup_{\stackrel{c_1\cdots c_k\in L}{\textup{cyclically reduced}}} E_{c_k^{-1},c_1} E_{c_1^{-1},c_2}\cdots E_{c_{k-1}^{-1},c_k}.\label{Conc_k_el_bloc}
\end{equation}
It follows that the contribution to the standard growth series $\sgs_{(G\wr L,\vec{Y})}(z)$ of the elements in $\conjrep^A(G\wr L)$ is
\begin{align*}
 F_{\conjrep^A(G\wr L) }(z)=&\,F_{\geocl(L,X)^A}(F_E(z)).\\
 %=&\frac{1}{1-(2M+N-1)F_E(z)}+\frac{(2M+N)F_E(z)^2-(N-1)F_E(z)-1}{1-F_E(z)^2}.
\end{align*}
Using Lemma \ref{cyl_red_A} one deduces that
\begin{equation}
 F_{{\conjrep_k} ^A(G\wr L)}(z)=\left((2M+N-1)^k+(-1)^k(M+N-1)+M\right)(F_E(z))^k\label{F_conj_rep_k},
\end{equation}
and
\[
 F_{\conjrep^A(G\wr L) }(z)=\sum_{k\geq 1}F_{{\conjrep_k} ^A(G\wr L)}(z).
\]

\begin{proof}[Proof of Proposition \ref{Lem_cgs_A}]
 \begin{comment}
 Consider the elements $(\eta,c)$, with $c=c_1\cdots c_k$ freely reduced of infinite order, and where $\eta$ satisfies the condition 
 \[
  \supp(\eta)\subset T_0\cup c_1 T_1 \cup\ldots \cup c_1\cdots c_{k-1} T_{k-1}.
 \]
 Recall that for $r=\{0,\ldots,k-1\}$, $T_r$ is the set of (freely reduced) words in $X$ that do not begin with $c_r^{-1}$ or $c_{r+1}$, and $c_0:=c_k$.
 
The walk on $\Gra(L,X)$ corresponding to $(\eta,c)$ starts at the origin, goes into $T_0$, creates the non-trivial elements of $G$, then comes back to $e'$ before going into $c_1 T_1$, where again it creates the non-trivial elements of $G$, until it reaches $c_1\cdots c_{k-1}T_{k-1}$, it creates the non-trivial elements of $G$, comes back to $c_1\cdots c_{k-1}$ and finally ends at $c$.
As explained above we have to consider the set
\[
 \conjrep^A(G\wr L)=\bigcup_{\stackrel{c_1\cdots c_k\in L,\,k\geq 1}{\textup{cyclically reduced}}} E_{c_1,c_k^{-1}} E_{c_2,c_1^{-1}}\cdots E_{c_k,c_{k-1}^{-1}},
\]
where the $E_{x_1,x_2}$'s refer to (\ref{El_bloc}).
\end{comment}
All the elements of $\conjrep^A(G\wr L)$ are minimal in their conjugacy class, every element of type $A$ in $G\wr L$ is conjugate to an element in $\conjrep^A(G\wr L)$, and two elements in $\conjrep^A(G\wr L)$ are conjugate if and only if they differ by a cyclic permutation of the $E_{x_1,x_2}$ components.\\

For every $k\geq 1$, the cyclic group $C_k$ acts on the set
\[
 {\conjrep_k} ^A(G\wr L)=\bigcup_{\stackrel{c_1\cdots c_k\in L}{\textup{cyclically reduced}}} E_{c_k^{-1},c_1} E_{c_1^{-1},c_2}\cdots E_{c_{k-1}^{-1},c_k}
\]
by cyclic permutation.

Since for $k\neq k'$, $\quotient{{\conjrep_k} ^A(G\wr L)}{C_k}\cap\quotient{{\conjrep_{k'}} ^A(G\wr L)}{C_{k'}}=\emptyset$, 
\[
 \cgs^A _{G\wr L,\vec{Y}}(z)=\sum_{k\geq 1}F_{\quotient{{\conjrep_k} ^A(G\wr L)}{C_k}}(z).
\]
For every $k>1$ and every $m$, the action of $C_k$ on ${\conjrep_k} ^A(G\wr L)$ preserves
\[
 S^k(m):=\left\{ U\in{\conjrep_k} ^A(G\wr L) \,:\,|U|_{\vec{Y}}=m\right\}.
\]
So $[z^m]F_{\quotient{{\conjrep_k} ^A(G\wr L)}{C_k}}(z)$ is the number of orbits in $S^k(m)$ under the action of $C_k$. Using Burnside's Lemma we find
\[
 [z^m]F_{\quotient{{\conjrep_k} ^A(G\wr L)}{C_k}}(z)=\frac{1}{k}\sum_{r\in C_k}\card\fix(r)=\frac{1}{k}\sum_{d|k}\sum_{\stackrel{1\leq r\leq k}{(r,k)=d}}\card\fix(r),
\]
where $\fix(r)$ denotes the set of elements of $S^k(m)$ fixed by the class of $r$ in $C_k$.
In fact, for $d|k$, $1\leq r\leq k$ and $(r,k)=d$, $U\in\fix(r)$ if and only if $U=W^{\frac{k}{d}}$ for some $W\in{\conjrep_d} ^A(G\wr L) $ and hence $|W|_{\vec{Y}}=\frac{md}{k}$. So
\[
 \card\fix(r)=[z^{\frac{md}{k}}]F_{{\conjrep_d} ^A(G\wr L)}(z)=[z^m]F_{{\conjrep_d} ^A(G\wr L)}(z^{\frac{k}{d}} ).
\]
Therefore we find 
\[
 F_{\quotient{{\conjrep_k} ^A(G\wr L)}{C_k}}(z)=\frac{1}{k}\sum_{d|k}\phi(\frac{k}{d})F_{{\conjrep_d} ^A(G\wr L)}(z^{\frac{k}{d}}),
\]
and summing over all possible values of $k$ gives
\[
 \cgs^A _{(G\wr L,\vec{Y})}(z)=\sum_{k\geq 1}\sum_{d|k}\frac{1}{k}\phi(\frac{k}{d})F_{{\conjrep_d} ^A(G\wr L)}(z^{\frac{k}{d}})
 =\sum_{r\geq 1}\sum_{s\geq 1}\frac{\phi(r)}{rs}F_{{\conjrep_s} ^A(G\wr L)}(z^r).
\]
By (\ref{F_conj_rep_k}) we obtain the result.
\end{proof}

\begin{corollary}\label{equality_rad_conv_wr}
 If $L$ is infinite, then the radius of convergence of the conjugacy growth series $\cgs _{(G\wr L,\vec{Y})}(z)$ is the same as the radius of convergence of the standard growth series $\sgs _{(G\wr L,\vec{Y})}(z)$.
\end{corollary}
\begin{proof}
 Firstly, it is clear that $\RC(\cgs^A _{(G\wr L,\vec{Y})}(z))\geq \RC(\cgs _{(G\wr L,\vec{Y})}(z))\geq\RC(\sgs _{(G\wr L,\vec{Y})}(z))$, so it remains to prove that $\RC(\cgs^A _{(G\wr L,\vec{Y})}(z))\leq \RC(\sgs _{(G\wr L,\vec{Y})}(z))$. In [\cite{P_92}, Theorem 4.1], it is proved that $\RC(\sgs _{(G\wr L,\vec{Y})}(z))$ is the smallest positive value $t$ such that 
 \[
(2M+N-1)(F_{E}(t))=1.
 \]
Proposition \ref{positive_preim_1} and Remark \ref{RC_cycl_redu} show that $\RC\left(F_{\geocl(L,X)^A}(F_E(z))\right)$ is also equal to the smallest positive value $t$ such that 
 \[
(2M+N-1)(F_{E}(t))=1.
 \]
 {\bfseries Claim:} For every $m\in\N$, the following holds:
 \begin{align*}
  [z^m]\sum_{k\geq 1}\frac{F_{{\conjrep_k} ^A(G\wr L)}(z)}{k}\leq [z^m]\cgs^A _{(G\wr L,\vec{Y})}(z)\leq& [z^m]F_{\geocl(L,X)^A}(F_E(z)).
   \end{align*}
   %Here ${\conjrep_k} ^A(G\wr L)$ refers to (\ref{Conc_k_el_bloc}) and $\conjrep^A(G\wr L)$ refers to (\ref{Conc_el_bloc}).
And hence 
\[
 \RC\left(\sum_{k\geq 1}\frac{F_{{\conjrep_k} ^A(G\wr L)}(z)}{k}\right)\geq \RC(\cgs^A _{(G\wr L,\vec{Y})}(z))\geq \RC\left(F_{\geocl(L,X)^A}(F_E(z))\right).
\]
{\bfseries Proof of the claim:} The first inequality is due to the fact that for each $k$ there are at most $k$ different elements in $\conjrep_k ^A(G\wr L)$ that represent the same conjugacy class (counting up to cyclic permutation of order $k$). The second inequality follows from
  \begin{align*}
 [z^m]\cgs^A _{(G\wr L,\vec{Y})}(z) &\leq [z^m]\sum_{k\geq 1}F_{{\conjrep_k} ^A(G\wr L)}(z)\\
  =&[z^m]F_{\conjrep^A(G\wr L) }(z)\\
  =&[z^m]F_{\geocl(L,X)^A}(F_E(z)).
     \end{align*}
 
 The corollary will follow from the claim and the inequality below:
 \[
  \RC\left(\sum_{k\geq 1}\frac{F_{{\conjrep_k} ^A(G\wr L)}(z)}{k}\right)\leq \RC\left(\sum_{k\geq 1}F_{{\conjrep_k} ^A(G\wr L)}(z)\right).
 \]
\begin{comment} 
For simplicity let us write 
\[
 r_k:=(2M+N-1)^k+(-1)^k(M+N-1)+M.
\]
 Hence $F_{{\conjrep_k} ^A(G\wr L)}(z)=r_k (F_E(z))^k$, where $F_E(z)$ is given by (\ref{growth_el_bloc}).
 \end{comment}
Note that because $F_E(0)=0$, for every $m\in\N$ we have
\[
  [z^m]\sum_{k\geq 1}F_{{\conjrep_k} ^A(G\wr L)}(z)=[z^m]\sum_{k= 1}^mF_{{\conjrep_k} ^A(G\wr L)}(z)
\]
and
\[
 [z^m]\sum_{k\geq 1}\frac{F_{{\conjrep_k} ^A(G\wr L)}(z)}{k}=[z^m]\sum_{k= 1}^m\frac{F_{{\conjrep_k} ^A(G\wr L)}(z)}{k}.
\]
Hence
\begin{align*}
[z^m]\sum_{k\geq 1}\frac{F_{{\conjrep_k} ^A(G\wr L)}(z)}{k}=&[z^m]\sum_{k= 1}^m\frac{F_{{\conjrep_k} ^A(G\wr L)}(z)}{k} \\
							    \geq&[z^m]\frac{1}{m}\sum_{k= 1}^m F_{{\conjrep_k} ^A(G\wr L)}(z)=\frac{1}{m}[z^m]\sum_{k\geq 1}F_{{\conjrep_k} ^A(G\wr L)}(z).
\end{align*}
Therefore passing to limsup we find
\begin{align*}
 \limsup_{m\to\infty}\root m \of{[z^m]\sum_{k\geq 1}\frac{F_{{\conjrep_k} ^A(G\wr L)}(z)}{k}}\geq& \limsup_{m\to\infty}\root m \of{\frac{1}{m}[z^m]\sum_{k\geq 1}F_{{\conjrep_k} ^A(G\wr L)}(z)}\\
							      &= \limsup_{m\to\infty}\root m \of{\sum_{k\geq 1}F_{{\conjrep_k} ^A(G\wr L)}(z)}.
\end{align*}
This proves that 
\[
 \RC\left(\sum_{k\geq 1}\frac{F_{{\conjrep_k} ^A(G\wr L)}(z)}{k}\right)\leq \RC\left(\sum_{k\geq 1}F_{{\conjrep_k} ^A(G\wr L)}(z)\right),
\]
and finalizes the proof.

\end{proof}

\subsection{Contribution of conjugacy classes with cursor of finite order}\label{cursor_finite}
Recall that $\cgs_{(G\wr L,\vec{Y})}^B(z)$ denotes the contribution to the conjugacy growth series of $G\wr L$ of the conjugacy classes of type B (i.e. with cursors of finite order).

  We will write $\cgs_{(G\wr L,\vec{Y})}^{B^{\neq e'}}(z)$ for the contribution to $\cgs_{(G\wr L,\vec{Y})}^B(z)$ of the conjugacy classes of elements $(\eta,b)$, where $b\neq e'$, and $\cgs_{(G\wr L,\vec{Y})}^{e'}(z)$ for the contribution to $\cgs_{(G\wr L,\vec{Y})}^B(z)$ of the conjugacy classes of elements of the form $(\eta,e')$. Hence $\cgs_{(G\wr L,\vec{Y})}^B(z)=\cgs_{(G\wr L,\vec{Y})}^{B^{\neq e'}}(z)+\cgs_{(G\wr L,\vec{Y})}^{e'}(z)$. We consider these two cases separately.

\subsubsection{Contribution of conjugacy classes with torsion cursor}\label{Sec_cgs_Bneqe} 
 
 \begin{proposition} \label{Lem_cgs_Bneqe}
 The growth series $\cgs_{(G\wr L,\vec{Y})}^{B^{\neq e'}}(z)$ is given by
  
  \begin{equation}
  \cgs_{(G\wr L,\vec{Y})}^{B^{\neq e'}}(z)=Nz\cgs_{(G,Y)}(z) \left(F_\T(z^2\cgs_{(G,Y)}(z),z^2(\cgs_{(G,Y)}(z)-1))\right)^{2M+N-1}\label{cgs_Bneqe}.
\end{equation}
  \end{proposition}
   \begin{proof}
   Let $(\eta,b)$ be an element in $G\wr L$, where $b\in L$ is a torsion element.
  By Lemma \ref{min_conj_length_wr_tree} one can assume that $b=b_j$ for some $j\in\{1,\ldots,N\}$. Let $T$ be the right transversal of $\langle b_j \rangle$ in $L$ consisting of the set of (freely reduced) words in $L$ that do not begin with $b_j$. By Corollary \ref{equiv_conj_non_triv_curs} B) one can assume that $\supp(\eta)\subset T$. Then an element $(\eta',b)$ with $\supp(\eta')\subset T$ is conjugate to $(\eta,b)$ in $G\wr L$ if and only if for every $l\in L$, $\eta'(l)$ is conjugate to $\eta(l)$ in $G$. 
  
  Hence we have to consider the walks on $\Gra(L,X)$ that go first in the $2M+N-1$ directions other than $b$, create the non-trivial conjugacy classes in $G$ of $\eta$, come back to the origin, and then finish at $b$ without creating any element at this position. So the contribution to $\cgs_{(G\wr L,\vec{Y})}^{B^{\neq e'}}(z)$ of the conjugacy classes of the elements $(\eta, b_j)$ for a given $j\in\{1,\ldots,N\}$ is given by
  \[
   z\cgs_{(G,Y)}(z)\left(F_\T(z^2\cgs_{(G,Y)}(z),z^2(\cgs_{(G,Y)}(z)-1))\right)^{2M+N-1}.
  \]
The term $\left(F_\T(z^2\sgs_{(G,Y)}(z),z^2(\sgs_{(G,Y)}(z)-1))\right)^{2M+N-1}$ counts the contribution of the walks starting at the origin, visiting every vertex of $\supp(\eta)$ (in the $2M+N-1$ directions away from $b_j$), creating the non-trivial conjugacy classes in $G$, and coming back to $e'$. The term $\cgs_{(G,Y)}$ counts the contribution at the component $e'$, and $z$ accounts for the step from $e'$ to $b$. Therefore summing over all the $j$'s, we find the result.
\end{proof}

 \subsubsection{Contribution of conjugacy classes with trivial cursor}\label{Sec_cgs_e}
We now explain how to compute $\cgs_{(G\wr L,\vec{Y})}^{e'}(z)$.
Lemma \ref{min_conj_length_wr_triv_curs} shows that a non-trivial element $(\eta,e')$ is minimal in its conjugacy class if and only if the tree spanned by $\supp(\eta)$ contains the origin $e'$ and all the components of $\eta$ are minimal in their conjugacy class (in $G$). Moreover, two such elements $(\eta,e')$, $(\eta',e')$ are conjugate if and only if there exists $d\in L$ such that for every $l\in L$, $\eta'(l)$ is conjugate (in $G$) to $\eta^d(l)$. 

Therefore we have to consider the following. For a finite subset $L'\subset L$, we write $\spanned(L')$ for the set of vertices of the tree spanned by $L'$, i.e. the smallest subset of $L$ containing $L'$ that forms a tree in $\Gra(L,X)$. Let $\Trees$ be the set 
\[
 \Trees:=\{L'\subset L\,:\,\card L'<\infty,\, e'\in L',\,L'=\spanned(L')\}.
\]
Although $L'\in \Trees$ denotes a set of vertices, we call it a tree, denote by $\Leafs(L')$ its set of leaves (vertices of degree $\leq 1$), and similarly write $\NLeafs(L'):=L'\setminus \Leafs(L')$ for its set of non-leaves.
Now let
\[
 \FuntreesG:=\bigcup_{L'\in \Trees}\{\eta:L'\rightarrow G:\,\forall l\in\Leafs(L'),\,\eta(l)\neq e\}.
\]
The set $\FuntreesG$ represents the collection of the labeled finite trees containing the origin with the leaves having non-trivial label. For an element $\eta:L'\rightarrow G$ in $\FuntreesG$ we define a weight $|\eta|_{\FuntreesG,\sim}\in\N\setminus\{0\}$ to be 
\[
 \eta|_{\FuntreesG,\sim}:=2\card\{\textup{ edges of the tree spanned by }L'\}+\sum_{l\in L'}|\eta(l)|_{\sim,Y}.
\]
Two functions $\eta_1:L_1\rightarrow G$ and $\eta_2: L_2\rightarrow G$ in $\FuntreesG$ are called equivalent if there exists $d\in L$ such that $L_1=d L_2$ and for every $l\in L_1$, $\eta_1(l)$ is conjugate in $G$ to $\eta_2(d^{-1}l)$. We denote this equivalence by $\sim$. Note that the weight $|\cdot|_{\FuntreesG,\sim}$ is preserved under this equivalence. Then there is a bijection between the set $\quotient{\FuntreesG}{\sim}$ and the set of non-trivial conjugacy classes $[(\eta,e')]$ in $G\wr L$, which for every $n$ restricts to the set of elements of $\quotient{\FuntreesG}{\sim}$ of weight $n$ and the set non-trivial conjugacy classes $[(\eta,e')]$ of conjugacy length $n$, i.e the following diagram commutes : 
\[
 \xymatrix{\quotient{\FuntreesG}{\sim}\ar@{<->}[rr] \ar[rdd]_{|\cdot|_{\FuntreesG,\sim}} & &\{[(\eta,e')]\in {(G\wr L)}_\sim\,:\,\eta\neq \vec{e}\}\ar[ddl]^{|\cdot|_{\sim,G\wr L}} \\
 &\circlearrowright & \\
    & \N\setminus\{0\}\rlap{\,.} &}
\]

One restricts the relation $\sim$ to $\Trees$ by defining that $L_1\sim L_2$ if there exists $d\in L$ such that $L_1=dL_2$. Let $\Treesbylefac\subset\Trees$ be a set of representatives of the elements in $\Trees$ with respect to $\sim$. 

For any $L'\in\Treesbylefac$, the subgroup of $L$ fixing $L'$ by left multiplication is finite, since $L'$ is finite, hence this subgroup is either trivial or isomorphic to $C_2$ (generated by a conjugate of some $b_j$); in the former case we call $L'$ {\em symmetric}, and in the latter {\em asymmetric}.
Note that since an element $L'$ of $\Treesbylefac$ contains $e'$, it is symmetric if and only if there exist $L_1\subset L'$, $h\in L_1$ and $j\in\{1,\ldots,N\}$ such that $L'=L_1\sqcup hb_jh^{-1}L_1$. In this case $L_1$, $h$ and $j$ are unique, and $\langle hb_j h^{-1} \rangle$ is the subgroup of $L$ fixing $L'$ by left multiplication; the element $h$ is the vertex where $L_1$ connects to $hb_jh^{-1}L_1$, and the following holds:
\begin{equation}
 \left\{\begin{array}{c}
         \card\Leafs(L_1\sqcup hb_jh^{-1}L_1)=2(\card\Leafs(L_1)-1)\quad\textup{and}\quad \card\NLeafs(L_1\sqcup hb_jh^{-1}L_1)=2(\card\NLeafs(L_1)+1)\\
         \textup{if } L_1\neq\{e'\},\textup{ while}\\
         \card\Leafs(\{e,b_j\}=2\quad\textup{and}\quad \card\NLeafs(\{e,b_j\})=0\\
         \textup{ for } j\in\{1,\ldots,M\}
        \end{array}\right.\label{leaves_and_nonleaves}
\end{equation}

Note that the number of edges in the tree spanned by $ L'\in\Treesbylefac$ is $\card L'-1$.

By Burnside's Lemma we have the following.
\begin{proposition}\label{Lem_cgs_Be}
The contribution to conjugacy growth series of the conjugacy classes of the elements with trivial cursor is given by
\begin{align*}
 \cgs_{(G\wr L,\vec{Y})}^{e'}(z)=&\,1+\sum_{\textup{asymmetric }L'\in\Treesbylefac }z^{2(\card L'-1)}{\left(\cgs_{(G,Y)}(z)-1\right)}^{\card\Leafs(L')}{\cgs_{(G,Y)}(z)}^{\card\NLeafs(L')}\\  
   &+\frac{1}{2}\sum_{\textup{symmetric } L'=L_1\sqcup hb_jh^{-1}L_1\in\Treesbylefac}z^{2(\card L'-1)} {\left(\cgs_{(G,Y)}(z)-1\right)}^{\card\Leafs(L')}{\cgs_{(G,Y)}(z)}^{\card\NLeafs(L')}\\
   &+\frac{1}{2}M z^2(\cgs_{(G,Y)}(z^2)-1)\\
   &+\frac{1}{2}\sum_{\underset{\,L_1\neq\{e\}}{\textup{symmetric } L'=L_1\sqcup hb_jh^{-1}L_1\in\Treesbylefac}}z^{2(\card L'-1)} {\left(\cgs_{(G,Y)}(z^2)-1\right)}^{\card\Leafs(L_1)-1}{\cgs_{(G,Y)}(z^2)}^{\card\NLeafs(L_1)+1}.\label{fat}
\end{align*}
\end{proposition}
The first term, $1$, is the contribution of the trivial conjugacy class. The second term comes from the contribution of the asymmetric trees in $\Treesbylefac$, the term in the second line comes from the contribution of the symmetric trees in $\Treesbylefac$ fixed by the identity, and the terms in the third and the forth line come from the contribution of the symmetric trees in $\Treesbylefac$ fixed by the elements of order 2. According to (\ref{leaves_and_nonleaves}) we distinguish the $L_1$'s being $\{e\}$ or not.\\

Now the difficulty of computing $\cgs_{(G\wr L,\vec{Y})}^{e'}(z)$ resides in finding a suitable set $\Treesbylefac$.\\

If $L$ is free then there exists an explicit left order $<$ on $L$ \cite{Su_13}, and then for a given $L'\in\Trees$ there exists a unique $dL'\in \Trees$ (for $d\in L$) such that every element of $dL'$ is greater than or equal to $e'$ with respect to $<$.
\begin{lemma}[\cite{Su_13}:Theorem 1.2]
 If $L=\langle a_1,\ldots,a_M|\rangle$ is a free group then a left order $>$ on $L$ can be defined by
\[
 w>e\iff \left\{\begin{array}{c}
                 \card \{ a_j a_i^{-1}\textup{in }w,\, j > i \}>\card \{ a_j^{-1} a_i\textup{ in }w,\, j > i \} \\\textup{or}\\
                 \card \{ a_j a_i^{-1}\textup{ in }w,\, j > i \}=\card \{ a_j^{-1} a_i \textup{ in }w , j > i \}\\ 
                 \textup{ and }w\textup{ ends with }a_i,\,i\in\{1,\ldots,M\} \rlap{\,.}                 
                \end{array}\right. 
\]
\end{lemma}
So if $L$ is free one can define $\Treesbylefac$ as
\[
 \Treesbylefac:=\{L'\in\Trees\,:\,\forall l\in L'\setminus\{e'\},\,l>e\},
\]
and no element of $\Treesbylefac$ will be symmetric (from our last definition).
It is a classic result that free groups are left-orderable, dating all the way to \cite{V49}, and we explicitly mention the left order given by {\v{S}}uni{\'c} above because of its simplicity and with the hope that it can exploited to find formulas for the conjugacy growth series in this setting.

Although we were unable to find a general formula for the conjugacy series of groups of the form $G \wr L$, where $\Gra(L,X)$ is a tree, we observe that $\cgs_{(G\wr L,\vec{Y})}^{e'}(z)$ is a complex power series with natural coefficients in the variables $z\cgs_{(G,Y)}(z)$, $z(\cgs_{(G,Y)}(z)-1)$, $z^{2}\cgs_{(G,Y)}(z^2)$ and $z^{2}(\cgs_{(G,Y)}(z^2)-1)$, or just $z^{2}\cgs_{(G,Y)}(z)$ and $z^{2}(\cgs_{(G,Y)}(z)-1)$ when $L$ is a free group.

\section{Examples}\label{Sec_examples}

In this section we give explicit computations of the conjugacy growth series for several examples.

\begin{example}[Conjugacy growth series of $G\wr C_2$]

Let $C_2=\langle b|b^2\rangle$. Since in $C_2$ there are no elements of infinite order, $\cgs_{(G\wr C_2,\vec{Y})} ^A(z)=0$. The series $\cgs_{(G\wr L,\vec{Y})}^{B^{\neq e'}}(z)$ is 
\[
 \cgs_{(G\wr C_2,\vec{Y})}^{B^{\neq e'}}(z)=z\cgs_{(G,Y)}(z).
\]
Now the series $\cgs_{(G\wr L,\vec{Y})}^{e'}(z)$ is the sum of one term counting the contribution of the classes of elements $(\eta,e')$ when $\supp(\eta)$ is a singleton or the empty set, which is equal to $\cgs_{(G,Y)}(z)$, and the contribution of the conjugacy classes of elements $(\eta,e')$ with $\supp(\eta)=C_2$. Such a conjugacy class is unique up to permutation by $C_2$ of the components of $\eta$ and up to conjugation in $G$. So using Burnside's Lemma we find
\[
 \cgs_{(G\wr L,\vec{Y})}^{e'}(z)=\cgs_{(G,Y)}(z)+\frac{z^2}{2}\left({(\cgs_{(G,Y)}(z)-1)}^2+\cgs_{(G,Y)}(z^2)-1\right).
\]
Therefore
\[
 \cgs_{(G\wr C_2,\vec{Y})}(z)= \cgs_{(G,Y)}(z)+z\cgs_{(G,Y)}(z)+\frac{z^2}{2}\left({(\cgs_{(G,Y)}(z)-1)}^2+\cgs_{(G,Y)}(z^2)-1\right).
\]
Note that the radius of convergence of $\cgs_{(G\wr C_2,\vec{Y})}(z)$ is the same as the radius of convergence of $\cgs_{(G,Y)}(z)$.
\end{example}
\begin{example}[Conjugacy growth series of $G\wr\Z$]
Let $\Z=\langle a|\rangle$. Using formula (\ref{cgs_A}) in Section \ref{Sec_cgs_A} with $M=1$ and $N=0$ we find
\[
 \cgs_{G\wr\Z,\vec{Y}}^A(z)=2\sum_{r\geq1}\frac{\phi(r)}{r}\sum_{s\geq 1}\frac{z^{rs}{\sgs_{(G,Y)}(z^r)}^s}{s}.
\]
Since in $\Z$ there are no torsion elements, $\cgs_{(G\wr \Z,\vec{Y})}^{B^{\neq e'}}(z)=0$. Now for $\cgs_{(G\wr L,\vec{Y})}^{e'}(z)$ notice that if $\supp(\eta)\neq\emptyset$, then there exists a unique translate $\eta'$ of $\eta$ (i.e. $\eta'=\eta^d$ for a certain $d\in\Z$) such that $\supp(\eta')\subset\N$ and $\eta'(e')\neq e$. Hence we must count for every subset $\{0,\ldots,n\}\subset\N$, the contribution of the classes of elements $(\eta,0)$ with $\supp(\eta)\subset\{0,\ldots,n\}$, $\eta(0)\neq e$ and $\eta(n)\neq e$. We finally find
\[
 \cgs_{(G\wr L,\vec{Y})}^{e'}(z)=\cgs_{(G,Y)}(z)+z^2{(\cgs_{(G,Y)}(z)-1)}^2\sum_{k=0}^\infty {(z^2\cgs_{(G,Y)}(z))}^k=\cgs_{(G,Y)}(z)+\frac{z^2{(\cgs_{(G,Y)}(z)-1)}^2}{1-z^2\cgs_{(G,Y)}(z)}.
\]
Therefore
\[
 \cgs_{G\wr\Z,\vec{Y}}(z)=2\sum_{r\geq1}\frac{\phi(r)}{r}\sum_{s\geq 1}\frac{z^{rs}{\sgs_{(G,Y)}(z^r)}^s}{s}+\cgs_{(G,Y)}(z)+\frac{z^2{(\cgs_{(G,Y)}(z)-1)}^2}{1-z^2\cgs_{(G,Y)}(z)}.
\]
The radius of convergence of $\cgs_{G\wr\Z,\vec{Y}}(z)$ is the unique positive value $t$ such that $t\sgs_{(G,Y)}(t)=1$.
\end{example}
\begin{example}[Conjugacy growth series of the Lamplighter group $C_2\wr \Z$]
Let $G=C_2$ be the group with 2 elements. Using the formula above we find
\[
 \cgs_{(C_2\wr \Z,\vec{Y})}(z)=2\sum_{r\geq1}\frac{\phi(r)}{r}\sum_{s\geq 1}\frac{z^{rs}{(1+z^r)}^s}{s}+1+z+\frac{z^4}{1-z^2(1+z)}.
\]
We give below an asymptotic estimate of the conjugacy growth $[z^m]\cgs_{(C_2\wr \Z,\vec{Y})}(z)$. We say that two functions $f,g:\,\N\rightarrow\N$ are asymptotically equivalent, and write $f\sim_{\textup{as}} g$, if 
\[
 \lim_{m\to\infty}\frac{f(m)}{g(m)}=1.
\]
Since 
\[
 \RC\left(2\sum_{r\geq1}\frac{\phi(r)}{r}\sum_{s\geq 1}\frac{z^{rs}{(1+z^r)}^s}{s}\right)<\RC\left(1+z+\frac{z^4}{1-z^2(1+z)}\right)
\]
it follows that 
\[
[z^m]\cgs_{(C_2\wr \Z,\vec{Y})}(z)\sim_{\textup{as}} [z^m]2\sum_{r\geq1}\frac{\phi(r)}{r}\sum_{s\geq 1}\frac{z^{rs}{(1+z^r)}^s}{s}.
\]
We have the following.
\begin{align*}
 [z^m]2\sum_{r\geq1}\frac{\phi(r)}{r}\sum_{s\geq 1}\frac{z^{rs}{(1+z^r)}^s}{s}&=2\sum_{\stackrel{(r,s,k)\in\N^{*}\times\N^{*}\times\N}{r(s+k)=m,\,k\leq s}}\frac{\phi(r)}{rs}\binom{s}{k}\\
 &=2\sum_{r|m}\frac{\phi(r)}{r}\sum_{s=\frac{m}{2r}}^{\frac{m}{r}}\frac{1}{s}\binom{s}{\frac{m}{r}-s}\\
 &\stackrel{\text{\footnotemark{}}}{=}2\sum_{r|m}\frac{\phi(r)}{m}\left({\left(\frac{-1+\sqrt{5}}{2}\right)}^{\frac{m}{r}}+{\left(\frac{1+\sqrt{5}}{2}\right)}^{\frac{m}{r}}\right)\\
 &=2\frac{1}{m}\left({\left(\frac{-1+\sqrt{5}}{2}\right)}^{m}+{\left(\frac{1+\sqrt{5}}{2}\right)}^{m}\right)\\
 &+2\sum_{\stackrel{r|m}{r<m}}\frac{\phi(r)}{m}\left({\left(\frac{-1+\sqrt{5}}{2}\right)}^{\frac{m}{r}}+{\left(\frac{1+\sqrt{5}}{2}\right)}^{\frac{m}{r}}\right)\\
 &\sim_{\textup{as}}\frac{2}{m}\left({\left(\frac{-1+\sqrt{5}}{2}\right)}^{m}+{\left(\frac{1+\sqrt{5}}{2}\right)}^{m}\right)\\
 &\sim_{\textup{as}}\frac{2}{m}{\left(\frac{1+\sqrt{5}}{2}\right)}^{m}
\end{align*}
Therefore \footnotetext{We use the software Wolframalpha.}
\begin{equation}
 [z^m]\cgs_{(C_2\wr \Z,\vec{Y})}(z)\sim_{\textup{as}}\frac{2}{m}{\left(\frac{1+\sqrt{5}}{2}\right)}^{m}\label{estimateL2}.
\end{equation}

\end{example}
\begin{proposition}\label{tran_L2}
 The conjugacy growth series $\cgs_{(C_2\wr \Z,\vec{Y})}(z)$ of the Lamplighter group is transcendental over $\Q(z)$.
\end{proposition}
\begin{proof}
This follows from the estimate (\ref{estimateL2}) and \cite[Theorem D]{F87}.  
\end{proof}

\begin{example}[Conjugacy growth series of $G\wr(C_2*C_2)$]
Let $C_2*C_2=\langle b_1,b_2|b_1^2,b_2^2\rangle$. Using formula (\ref{cgs_A}) in Section \ref{Sec_cgs_A} with $M=0$ and $N=2$ we find
\[
 \cgs_{(G\wr (C_2*C_2),\vec{Y})}^A(z)=\sum_{r\geq 1}\frac{\phi(r)}{r}\sum_{s\geq 1}\frac{z^{2rs}{\sgs_{(G,Y)}(z^r)}^{2s}}{s}.
\]
Using formula (\ref{cgs_Bneqe}) in Section \ref{Sec_cgs_Bneqe} with $M=0$ and $N=2$, we find
\[
  \cgs_{(G\wr (C_2*C_2),\vec{Y})}^{B^{\neq e'}}(z)=2z\cgs_{(G,Y)}(z)\frac{1-z^2}{1-z^2\cgs_{(G,Y)}(z)}.
\]
For the computation of $\cgs_{G\wr (C_2*C_2)}^{e'}(z)$ we proceed as follows. Let us consider the set $\Trees$ as in Section \ref{Sec_cgs_e}. Every such element spans a line. Let $L'\in \Trees$ and let $r$ be the diameter of $L'$. There are two ways, according to the parity of $r$, to choose a representative $L'$ under the relation given by the left action of $L$.
\begin{enumerate}[$\bullet$]
 \item If $r$ is even there exists a unique $\tilde{L}\subset\{\textup{words that do not begin with }b_2\}$ that is a left translate of $L'$ in $\Trees$. The subgroup of $L$ fixing $\tilde{L}$ by left multiplication is trivial.
 Hence the contribution of all such elements is
 \[
  \cgs_{(G,Y)}(z)-1+{\left(\cgs_{(G,Y)}(z)-1\right)}^2\sum_{0<r\textup{ even}}z^{2r}{\cgs_{(G,Y)}(z)}^{r-1}=\cgs_{(G,Y)}(z)+{\left(\cgs_{(G,Y)}(z)-1\right)}^2\frac{z^4\cgs_{(G,Y)}(z)}{1-z^4{\cgs_{(G,Y)}(z)}^2}\rlap{\,.}
 \]
 \item If $r$ is odd then exactly one of the following happens:
 \begin{enumerate}[i)]
  \item There exists a unique $\tilde{L}\subset\{\textup{words that do not begin with }b_2\}$ that is a left translate of $L'$ in $\Trees$, or
  \item there exists a unique $\tilde{L}\subset\{b_2\}\cup\{\textup{words that do not begin with }b_2\}$ that is a left translate of $L'$ in $\Trees$.
 \end{enumerate}
In both cases the subgroup of $L$ fixing $\tilde{L}$ by left multiplication is $C_2$. Hence each of the cases i) and ii) gives rise to
\begin{align*}
 &\frac{1}{2}\left({(\cgs_{(G,Y)}(z)-1)}^2 \sum_{0<r\textup{ odd}}z^{2r}{\cgs_{(G,Y)}(z)}^{r-1}+(\cgs_{(G,Y)}(z^2)-1)\sum_{0<r\textup{ odd}}z^{2r}{\cgs_{(G,Y)}(z^2)}^{\frac{r-1}{2}}\right)\\
 &=\frac{z^2}{2}\left(\frac{{(\cgs_{(G,Y)}(z)-1)}^2}{1-{z^4\cgs_{(G,Y)}(z)}^2}+\frac{\cgs_{(G,Y)}(z^2)-1}{1-z^4\cgs_{(G,Y)}(z^2)}\right).
\end{align*}

\end{enumerate}
Therefore summing the contribution of the even and odd $r$'s and the contribution of the trivial conjugacy class we find
\[
 \cgs_{G\wr(C_2*C_2),\vec{Y}} ^{e'}(z)=\cgs_{(G,Y)}(z)+z^2\left(1+z^2{\cgs_{(G,Y)}(z)}\right)\frac{({\cgs_{(G,Y)}(z)-1)}^2}{1-z^4{\cgs_{(G,Y)}(z)}^2}+ z^2\frac{(\cgs_{(G,Y)}(z^2)-1)}{1-z^4\cgs_{(G,Y)}(z^2)}.
\]
Finally
\begin{align*}
\cgs_{G\wr(C_2*C_2),\vec{Y}}(z)= &\,2\sum_{r\geq 1}\frac{\phi(r)}{r}\sum_{s\geq 1}z^{2rs}{\sgs_{(G,Y)}(z^r)}^{2s}+2z\cgs_{(G,Y)}(z)\frac{1-z^2}{1-z^2\cgs_{(G,Y)}(z)}\\
&+ \cgs_{(G,Y)}(z)+z^2\left(1+z^2{\cgs_{(G,Y)}(z)}\right)\frac{({\cgs_{(G,Y)}(z)-1)}^2}{1-z^4{\cgs_{(G,Y)}(z)}^2}+ z^2\frac{(\cgs_{(G,Y)}(z^2)-1)}{1-z^4\cgs_{(G,Y)}(z^2)}.
\end{align*}
\end{example}

\section*{Acknowledgments}
The author was supported by the Swiss National Science 
Foundation grant Professorship FN PP00P2-144681/1. He would like to thank his PhD advisor Laura Ciobanu for her valuable help on this paper, Rémi Coulon for helpful discussions, and Ana Khukhro for her proofreading.

\bigskip

\textsc{V. Mercier,
Mathematics Department,
University of Neuch\^atel,
Rue Emile-Argand 11,
CH-2000 Neuch\^atel, Switzerland
}

\emph{E-mail address}{:\;\;}\texttt{valentin.mercier@unine.ch}

\end{document}